\documentclass[reqno,11pt,onesided]{article}
\usepackage{amsmath,amsthm,amssymb}
\usepackage{verbatim}

\date{}

 \newcommand{\norm}[1]{\left\Vert#1\right\Vert}

\newcommand{\abs}[1]{\left\vert#1\right\vert}
\newcommand{\set}[1]{\left\{#1\right\}}
\newcommand{\brac}[1]{\left(#1\right)}
\newcommand{\scalar}[1]{\left \langle #1 \right \rangle}

\newcommand{\Real}{\mathbb{R}}
\newcommand{\E}{\mathbb{E}}

\newcommand{\eps}{\epsilon}

\newcommand{\Ric}{\mbox{\rm{Ric}}}
\newcommand{\vol}{\mbox{\rm{vol}}}
\newcommand{\II}{\text{II}}

\newcommand{\Var}{\text{\rm{Var}}}
\newcommand{\Ent}{\text{\rm{Ent}}}
\newcommand{\Hess}{\text{Hess}}
\DeclareMathOperator{\intr}{int}
\DeclareMathOperator{\supp}{supp}
\DeclareMathOperator{\Img}{Im}

\def\XXint#1#2#3{{\setbox0=\hbox{$#1{#2#3}{\int}$}
\vcenter{\hbox{$#2#3$}}\kern-.5\wd0}}

\newtheorem{thm}{Theorem}[section]
\newtheorem*{thm*}{Theorem}
\newtheorem{cor}[thm]{Corollary}
\newtheorem{lem}[thm]{Lemma}
\newtheorem{prop}[thm]{Proposition}
\newtheorem*{prop*}{Proposition}

\newtheorem*{problem*}{Problem}

\newtheorem*{conj*}{Conjecture}
\newtheorem{defn}[thm]{Definition}
\newtheorem*{dfn*}{Definition}
\theoremstyle{definition}
\newtheorem{rem}[thm]{\textbf{Remark}}
\newtheorem*{rmk*}{Remark}
\newtheorem*{fact*}{Fact}
\newtheorem{example}[thm]{Example}
\theoremstyle{proof}

\numberwithin{equation}{section}

\begin{document}

\title{Riemannian metrics on convex sets with applications to Poincar\'e and log-Sobolev inequalities}
\author{Alexander V. Kolesnikov\textsuperscript{1} and Emanuel Milman\textsuperscript{2}}
\footnotetext[1]{Faculty of Mathematics, Higher School of Economics, Moscow, Russia. 
This work was supported by the Russian Science Foundation grant 14-11-00196  at the Lomonosov Moscow State University. E-mail: Sascha77@mail.ru.}
\footnotetext[2]{Department of Mathematics, Technion - Israel
Institute of Technology, Haifa 32000, Israel. The research leading to these results is part of a project that has received funding from the European Research Council (ERC) under the European Union's Horizon 2020 research and innovation programme (grant agreement No 637851). In addition, E.M. was supported by ISF (grant no. 900/10), BSF (grant no. 2010288) and Marie-Curie Actions (grant no. PCIG10-GA-2011-304066). Email: emilman@tx.technion.ac.il.}

\begingroup    \renewcommand{\thefootnote}{}    
\endgroup

\maketitle

\begin{abstract}
Given a probability measure $\mu$ supported on a convex subset $\Omega$ of Euclidean space $(\Real^d,g_0)$, we are interested in obtaining Poincar\'e and log-Sobolev type inequalities on $(\Omega,g_0,\mu)$. To this end, we change the metric $g_0$ to a more general Riemannian one $g$, adapted in a certain sense to $\mu$, and perform our analysis on $(\Omega,g,\mu)$. The types of metrics we consider are Hessian metrics (intimately related to associated optimal-transport problems), product metrics (which are very useful when $\mu$ is unconditional, i.e. invariant under reflections with respect to the coordinate hyperplanes), and metrics conformal to the Euclidean one, which have not been previously explored in this context. Invoking on $(\Omega,g,\mu)$ tools such as Riemannian generalizations of the Brascamp--Lieb inequality and the Bakry--\'Emery criterion, and passing back to the original Euclidean metric, we obtain various weighted inequalities on $(\Omega,g_0,\mu)$: refined and entropic versions of the Brascamp--Lieb inequality, weighted Poincar\'e and log-Sobolev inequalities, Hardy-type inequalities, etc. 
Key to our analysis is the positivity of the associated Lichnerowicz--Bakry--\'Emery generalized Ricci curvature tensor, and the convexity of the manifold $(\Omega,g,\mu)$. In some cases, we can only ensure that the latter manifold is (generalized) mean-convex, resulting in additional boundary terms in our inequalities. 
 \end{abstract}

\section{Introduction}
The starting point of this work is the classical Poincar\'e-type inequality established by H.~Brascamp and E.~Lieb in \cite{BrascampLiebPLandLambda1}. It states that for any probability measure $\mu$ with smooth positive density $\exp(-V)$ on $\Real^d$, one has
\begin{equation} \label{eq:classical-BL}
D^2 V > 0  \;\;\; \Rightarrow \;\;\; \Var_\mu(f) \leq \int_{\Real^d} \scalar{ (D^2 V)^{-1} \nabla f, \nabla f} d\mu \;\;\; \forall f \in C^1(\Real^d) . 
\end{equation}
Here $\Var_\mu(f) = \int f^2 d\mu - (\int f d\mu)^2$ denotes the variance of $f$ under $\mu$. 

\medskip

The original Brascamp--Lieb inequality (\ref{eq:classical-BL}) has since been generalized and extended (e.g. \cite{Helffer-DecayOfCorrelationsViaWittenLaplacian, LedouxSpinSystemsRevisited,BobkovLedouxWeightedPoincareForHeavyTails, NguyenDimensionalBrascampLieb, KolesnikovEMilman-Reilly}) to a much more general \emph{weighted Riemannian manifold setting}. 
 A triplet $(M,g,\mu)$ is called a weighted manifold if $(M,g)$ is a smooth connected Riemannian manifold and $\mu$ is a probability measure on $M$ having smooth and positive density $\exp(-V)$ with respect to the Riemannian volume measure $\vol_g$. Its Lichnerowicz--Bakry--\'Emery generalized Ricci curvature tensor $\Ric_{g,\mu}$ \cite{Lichnerowicz1970GenRicciTensor,BakryEmery} is defined as 
\[
\Ric_{g,\mu} := \Ric_g + D_g^2 V ,
\]
where $\Ric_g$ denotes the Ricci curvature tensor and $D^2_g$ denotes the Riemannian Hessian operator. 
The simplest version of the \emph{generalized} Brascamp--Lieb inequality states that, under certain \emph{convexity} (and other technical) assumptions on $(M,g)$,\begin{equation} \label{eq:BL}
\Ric_{g,\mu} > 0 \;\;\; \Rightarrow \;\;\;  \Var_\mu(f) \leq \int_M \scalar{ \Ric_{g,\mu}^{-1} \nabla f, \nabla f } d \mu \;\;\;  \forall f \in C^1(M)  .
\end{equation} 
Here $\scalar{\cdot,\cdot}$ denotes the Riemannian metric $g$ and $\nabla = \nabla_g$ is the Riemannian Levi-Civita connection. 
We refrain at the moment from explicitly stating the aforementioned convexity assumptions, and simply collectively refer to them as the property that ``$M$ is $g$-convex". For instance, the convexity assumptions are satisfied when $(M,g)$ is oriented, complete with smooth locally-convex boundary, or when $(M,g)$ is geodesically convex without boundary; the most general conditions are formulated in Section \ref{sec:pre}, with the corresponding proof of (\ref{eq:BL}) deferred to the Appendix.

An insightful observation which we exploit throughout the paper is the following:
(\ref{eq:BL}) is in fact the Poincar{\'e} inequality for the Riemannian metric
$$
\tilde{g} = \Ric_{g,\mu} .
$$
More generally, assume that we are given another (in our applications, Euclidean) metric $g_0$, and that
$\Ric_{g,\mu}$ satisfies
$$
\Ric_{g,\mu}\ge \lambda g_0 \;\; , \;\; \lambda > 0 . 
$$
Then (\ref{eq:BL}) implies the following Poincar{\'e} (or spectral-gap) inequality on $(M,g_0,\mu)$:
\begin{equation} \label{eq:intro-Poincare}
\Var_\mu(f) \le \frac{1}{\lambda} \int_M \bigl| \nabla_{g_0} f \bigr|^2_{g_0} \ d \mu \;\;\; \forall f \in C^1(M) .
\end{equation}
Reversing the reasoning, this suggests  the following approach for estimating the Poincar\'e constant of a given weighted manifold $(M,g_0,\mu)$: find a metric $g$ on $M$ so that $M$ is $g$-convex and so that $\Ric_{g,\mu} \geq \lambda g_0$ with $\lambda > 0$ as large as possible, and apply the Brascamp-Lieb inequality (\ref{eq:BL}), yielding the spectral-gap estimate (\ref{eq:intro-Poincare}) for $(M,g_0,\mu)$.
Moreover, if $\Ric_{g_0,\mu} \geq 0$ and $M$ is $g_0$-convex, then applying the known equivalence between concentration and isoperimetry on weighted manifolds with non-negative generalized Ricci curvature \cite{EMilman-RoleOfConvexity,EMilmanGeometricApproachPartI,EMilmanIsoperimetricBoundsOnManifolds}, it is in fact enough to control the harmonic mean: \[ \Var_{\mu}(f) \le C \int_M \frac{1}{\lambda} \ d \mu  \int_M \bigl| \nabla_{g_0} f \bigr|^2_{g_0} \ d \mu ,
\] where $\lambda : M \rightarrow (0,\infty)$ is any positive function such that $\Ric_{g,\mu} \geq \lambda g_0$, and $C \geq 1$ is a universal numeric constant. When $g_0$ is Euclidean, note that the condition $\Ric_{g_0,\mu} \geq 0$ translates into the assumption that $\mu = \exp(-V) dx$ is log-concave, namely that $V$ is convex. 

\medskip

There are  various natural  ways of constructing a metric $g$ as required above. However, a substantial obstacle arising with this strategy is the potential \emph{non-completeness} of the metric $g$, \emph{non-convexity} of the boundary of $M$ in this metric, and even \emph{non-geodesic-convexity} of $M$ itself. To this end, we exploit in this work a recent generalization of the Brascamp--Lieb inequality obtained in our previous work \cite{KolesnikovEMilman-Reilly}, where in particular a suitable modification of (\ref{eq:BL}) was established under the weaker assumption that $M$ is (generalized) mean-convex (see Sections \ref{sec:pre} and \ref{sec:conformal} for definitions). For concreteness, we summarize the above discussion in the Euclidean case as follows:

\begin{problem*} Given a smooth convex set $\Omega \subset \mathbb{R}^d$ with Euclidean metric $g_0$, let $\mu$ denote the uniform probability measure on $\Omega$. Find another Riemannian metric $g$ on $\Omega$ which minimizes $\int_\Omega \norm{\Ric_{g,\mu}^{-1}}_{op,g_0} \ d \mu$, provided that the weighted manifold $(\Omega, g, \mu)$ is $g$-convex or at the very least has generalized mean-convex boundary. 
\end{problem*}

Our main motivation here is the remarkable conjecture of Kannan, Lovasz and Simonovits \cite{KLS}, which predicts that when $(\Omega,g_0,\mu)$ is as above (and more generally, when $\mu$ is a log-concave measure), the best constant $\lambda > 0$ in the spectral-gap inequality (\ref{eq:intro-Poincare}) is obtained, up to a universal numerical constant $C$ (independent of the underlying dimension $d$), by solely testing linear functionals $f$. This conjecture has been confirmed for a handful of families of convex sets $\Omega$, mainly under extensive symmetry assumptions; confirming it for additional families or improving the best known (dimension-dependent) general estimates on $C$ are extremely challenging problems - see e.g. \cite{SpanishBook} and the references therein.  

\medskip

The above stated problem naturally leads to potentially very interesting (and difficult) variational problems on extremal metrics
(cf. \cite{Lions-FiniteDimKahlerEinstein}).
In this work we do not study this problem in full generality, but rather concentrate on 
several natural families of metrics which can provide good spectral-gap estimates.
Among them are the following ones (the symbols $D^2, \nabla$ are understood as derivatives with  respect to the standard Euclidean connection):
\begin{enumerate}
\item Hessian metrics of the form $g = D^2 \Phi$, where $\Phi$ is a strongly convex potential. These metrics naturally arise when $\mu$ is pushed-forward onto another probability measure $\nu$, by means of the optimal-transport map $x \mapsto \nabla \Phi(x)$ (see \cite{Kolesnikov-HessianMetricsArisingFromOT, KlartagMomentMap,Klartag-PayneWeinbergerViaMomentMeasures, KlartagKolesnikov-EigenvalueDistribution} for historical comments). In the past decades, the theory of optimal-transport has played a key role in the advancement of our understanding of isoperimetric and functional inequalities (see e.g. \cite{BogachevKolesnikov-OTSurvey,VillaniTopicsInOptimalTransport,VillaniOldAndNew,BGL-Book}), and it continues to play a natural role here as well. 
\item Product metrics of the form $g = \sum_{i=1}^d \Phi_i''(x_i) (dx^i)^2$. These metrics are very useful for studying \emph{unconditional} log-concave measures $\mu$, invariant under reflections with respect to the coordinate hyperplanes. 
\item Conformal metrics to the Euclidean one $g = e^{2 \varphi} \sum_{i=1}^d (dx^i)^2$. Conformal changes of metric are very important and useful transformations in Geometry and Mathematical Physics, which to the best of our knowledge, have not been previously exploited in our context (apart from very specific situations, like the known conformal relation between Euclidean space and the sphere, cf. \cite{BGL-Book}). 
\end{enumerate} 

For completeness, we also mention the possibility to use the Hessian metric arising from a potential given by the logarithmic Laplace transform of $\mu$, as initiated in \cite{EldanKlartagThinShellImpliesSlicing}, but this will not be further investigated in this work. 

\subsection{Refined Brascamp--Lieb Inequality}

In Section \ref{sec:refined-BL}, we obtain a refinement of the original Brascamp--Lieb inequality (\ref{eq:classical-BL}), by employing the Hessian metric $g = D^2 \Phi$ associated to an optimal-transport map $x \mapsto \nabla \Phi(x)$ pushing forward $\mu$ onto an auxiliary log-concave probability measure $\nu = \exp(-W) dx$. Choosing $\nu = \mu$ recovers the original Brascamp-Lieb inequality, and any other choice of log-concave measure $\nu$ recovers the original inequality up to a factor of $2$, suggesting that we might as well try and test other candidates $\nu$. As concrete examples, we analyze two cases: 
\begin{enumerate}
\item
Setting $\nu$ to be the uniform measure on some (any) convex set, we obtain a $(-d)$-dimensional version of the Brascamp--Lieb inequality for $\mu$,  previously considered by Bobkov--Ledoux \cite{BobkovLedouxWeightedPoincareForHeavyTails}, Nguyen \cite{NguyenDimensionalBrascampLieb} and the authors \cite{KolesnikovEMilman-Reilly} (however, we miss here the known sharp constant by a factor slightly less than $2$). 
\item
Given $\nu$, one may search for a probability measure $\mu = \exp(-V) dx$ so that the optimal-transport map $\nabla \Phi$ pushing forward $\mu$ onto $\nu$ satisfies $\Phi = V$. 
In that case, the resulting non-linear elliptic equation is the so-called (real-valued) K\"{a}hler--Einstein equation, investigated by Berman, Berndtsson, Cordero-Erausquin, Klartag and the first named author \cite{BermanBerndtsson-RealMongeAmpere, CorderoKlartag-MomentMeasures,KlartagMomentMap,Klartag-PayneWeinbergerViaMomentMeasures,KlartagKolesnikov-EigenvalueDistribution}. Building on their results, we obtain a version of the Brascamp--Lieb inequality for $\nu$ in the case it is supported inside a Euclidean ball of radius $R$, which interpolates between (\ref{eq:classical-BL}) and the Payne--Weinberger classical estimate \cite{PayneWeinberger} for the spectral-gap on a convex domain (extending Klartag's approach from \cite{Klartag-PayneWeinbergerViaMomentMeasures}). 
\end{enumerate}

\subsection{Entropic Brascamp--Lieb Inequality}

Besides obtaining Poincar\'e (\ref{eq:intro-Poincare}) or weighted Poincar\'e (\ref{eq:BL}) inequalities, we are also interested in obtaining (weighted) log-Sobolev inequalities. By the celebrated Bakry--\'Emery criterion \cite{BakryEmery} (see Section \ref{sec:pre} for missing definitions), whenever $M$ is $g$-convex and
\begin{equation} \label{eq:BE-cond}
\Ric_{g,\mu} \ge \rho g \;\; , \;\; \rho > 0 ,
\end{equation}
then $(M,g,\mu)$ satisfies the following log-Sobolev inequality:
\begin{equation} \label{EBL}
\Ent_{\mu}(f) \le \frac{2}{\rho}\int_M  \abs{\nabla_g f}^2 d \mu \;\;\; \forall f \in C^1(M) .
\end{equation}
Now when $\mu = \exp(-V) dx$ is a strongly log-concave measure supported on a convex subset $\Omega \subset \Real^d$, consider the Hessian metric $g = D^2 V$. If  $(\Omega,g = D^2 V , \mu = \exp(-V) dx)$ satisfies (\ref{eq:BE-cond}) and $\Omega$ is $g$-convex, it follows that
\begin{equation} \label{EBL0}
\Ent_\mu(f) \le \frac{2}{\rho}\int_\Omega \langle (D^2 V)^{-1} \nabla f, \nabla f \rangle d \mu \;\;\; \forall f \in C^1(\Omega),
\end{equation}
which is an entropic version of the original Brascamp--Lieb inequality (\ref{eq:classical-BL}). It was shown by S.~Bobkov and M.~Ledoux \cite{BobkovLedoux} that contrary to (\ref{eq:classical-BL}), (\ref{EBL0}) cannot hold for all (strongly) log-concave measures (with any $\rho > 0$). However, the above strategy allows us to obtain in Section \ref{sec:entropic-BL} a sufficient condition for (\ref{EBL0}) to hold  (which is different from the one derived in \cite{BobkovLedoux}). In particular, we obtain weighted log-Sobolev inequalities for the product measures $\exp(-\sum_{i=1}^d x_i^q) dx$ on $\Real_+^d$, $q \in (1,\infty)$, recovering when $q \in (1,2]$ a case treated by D.~Bakry in \cite{Bakry-JacobiSemigroups}. 

\subsection{Unconditional log-concave measures}

Non-trivial results may be obtained even by employing product metrics, the simplest amongst Hessian metrics. For instance, when $\mu = \exp(-V) dx$ is the conditioning of an unconditional log-concave probability measure onto the principle orthant $\Real_+^d$, we recover the following Poincar\'e-type inequality due to B.~Klartag \cite{KlartagMomentMap} (which implies the so-called thin-shell estimate for unconditional convex bodies):
\[
\Var_\mu(f) \le 4 \int \sum_{i=1}^d x^2_i f^2_{x_i} d \mu \;\;\; \forall f \in C^1(\Real_+^d) .
\]
Another classical result we recover is the modified log-Sobolev inequality for the $d$-dimensional (one-sided) exponential measure $\nu = \exp(-\sum_{i=1}^d x_i)$ on $\Real_+^d$ (see \cite{BobkovLedouxModLogSobAndPoincare}):
\[
\Ent_\nu (f^2) \le 4 \int \sum_{i=1}^d x_i  f^2_{x_i} d \nu \;\;\; \forall f \in C^1(\Real_+^d) .
\]
Generalizations of these weighted Poincar\'e and log-Sobolev inequalities to the case when $V_{x_i} \geq 0$ or $V_{x_i} \geq \lambda > 0$ are obtained in Section \ref{sec:unconditional}. It is also possible to transfer the resulting inequalities from unconditional log-concave measures to uniform measures on unconditional convex sets, assuming that the normal to their boundary is approximately diagonal (such as for the unit-ball of $\ell_1^d$) - see Theorem \ref{l1-type}.

\subsection{Conformal Metrics}

Conformal metrics enable us more flexibility on one hand, but the price we quite often pay is that the resulting manifold's boundary is no longer locally-convex. To address this issue, we employ a suitable generalization of the Brascamp--Lieb inequality for (generalized) mean-convex domains, obtained in our previous work \cite{KolesnikovEMilman-Reilly}. Moreover, we will need to use a dimensional version of the latter inequality, with a \emph{negative} generalized dimension $N < 0$. For simplicity, we consider the uniform probability measure $\lambda_\Omega$ on a convex set $\Omega \subset \Real^d$ with the standard Euclidean metric $g_0$. Equipping $\Omega$ with (a suitable regularization of) the metric $g = |x|^{-2\theta} g_0$, we verify that $\Ric_{g,\lambda_\Omega} \geq c d^2 g_0$ for some constant $c > 0$ and an appropriate choice of $\theta > 0$. We thus obtain in Section \ref{sec:conformal} a Hardy-type inequality with additional boundary term, of the form
\[
\frac{1 - N}{\mbox{\rm{Vol}}(\Omega)} \int_{\partial \Omega} 
\frac{1}{ \frac{d-N}{2} \frac{ \langle  x, n\rangle}{ \abs{x}^2}  -N H_0(x)   } (f -C)^2 d \mathcal{H}^{d-1}.
\]
Here $C \in \Real$ is an arbitrary constant,  $n$ is the outer unit normal to $\partial \Omega$, $H_0(x)$ is the Euclidean mean-curvature of $\partial \Omega$ at $x$, and $\mathcal{H}^{d-1}$ denotes the $(d-1)$--dimensional Hausdorff measure. For some other results on Hardy-type inequalities obtained by the authors using a different approach see \cite{KolesnikovEMilman-HardyKLS}. Under an additional condition on the Euclidean second-fundamental form of the boundary:
\[
\exists \theta \in (0,1) \;\;\; \II_{\partial \Omega, g_0} \ge  \theta \frac{\scalar{x, n}}{\abs{x}^2} \cdot \mbox{\rm{Id}}|_{\partial \Omega},
\]
it turns out that $\partial \Omega$ is locally-convex in the $g$ metric, yielding weighted Poincar\'e and log-Sobolev inequalities on $(\Omega,g_0,\lambda_{\Omega})$. 

\medskip

\noindent \textbf{Acknowledgments.}
We would like to thank Bo'az Klartag for stimulating discussions, and the anonymous referee for thoroughly reading the manuscript and providing comprehensive comments.

\section{Preliminaries} \label{sec:pre}

In this section, we collect various preliminaries required for this work. Certain formal statements, well-known to experts, are rigorously formulated (with the technical proofs deferred to the Appendix).  

\subsection{Weighted Riemannian Manifolds}

Let $M^d$ be a $d$-dimensional topological manifold, without or with boundary $\partial M$, whose interior $\intr(M) := M \setminus \partial M$ is endowed with a smooth differentiable structure and smooth Riemannian metric $g$. Let $\mu$ be a measure on $M$ having smooth and positive density $\exp(-V)$ with respect to the Riemannian volume measure $\vol_g$ on $\intr(M)$ (so that in particular $\mu(\partial M) = 0$). The triplet $(M,g,\mu)$ is called a weighted Riemannian manifold. 
The manifold may be compact or non-compact, bounded or unbounded, but throughout this work we will always assume that $M$ is connected and that $\mu$ is a probability measure.
Many of the tools we employ in this work require the following additional assumptions:

\begin{dfn*}[\textbf{Geometric Convexity Assumptions}]
Either:
\begin{enumerate}
\item
$(M,g)$ is a complete and oriented Riemannian manifold, $\partial M$ (if non-empty) is smooth and locally-convex, and the density $\exp(-V)$ remains smooth and positive up to and including the boundary ; or:
\item
$\intr(M)$ is geodesically convex. \end{enumerate}
\end{dfn*}
\noindent
Recall that $\partial M$ is called locally-convex if its second fundamental form $\II_{\partial M}$ is positive semi-definite, where $\II_{\partial M}(X,Y) = g(\nabla_X n,Y)$ for $X,Y \in T \partial M$, $\nabla = \nabla_g$ denotes the Levi-Civita connection and $n = n_g$ denotes the outer unit normal to $\partial M$. A Riemannian manifold is called geodesically convex if there exists a distance-minimizing geodesic connecting any two given points. It is known \cite{BishopInBayleRosales} that geodesic convexity implies local convexity of the boundary, but not vice-versa in general (even for complete manifolds). 

\medskip
A key role in this work will be played by the Lichnerowicz--Bakry--\'Emery generalized Ricci curvature tensor, which is defined as
\[
\Ric_{g,\mu} := \Ric_g + D_g^2 V ,
\]
where $\Ric_g$ denotes the usual (geometric) Ricci curvature tensor, and $D_g^2 = \mbox{Hess}_g$ is the Riemannian Hessian operator. It will be useful to express this operator in local coordinates:
\begin{equation}
\label{hessformula}
\mbox{Hess}_g (f)_{ij}= \partial^2_{x_i x_j} f - \Gamma^{k}_{ij} \partial_{x_k} f ,
\end{equation}
where the Christoffel symbols are given by 
\begin{equation}
\label{christoff}
\Gamma^{m}_{ij} = \frac{1}{2} g^{mk} \Bigl( \frac{\partial g_{ki}}{\partial x_j} + \frac{\partial g_{kj}}{\partial x_i} -  \frac{\partial g_{ij}}{\partial x_k}\Bigr).
\end{equation}

Throughout this work we employ Einstein summation convention. Given a $2$-covariant tensor $A_{\alpha,\beta}$, its inverse is the $2$-contravariant tensor $(A^{-1})^{\beta,\gamma}$ given by
\[
A_{i,j} (A^{-1})^{j,k} = \delta_i^k . 
\]
As usual, the (contravariant) inverse $g^{-1}$ to the (covariant) metric $g = g_{\alpha,\beta}$ is also denoted by $g = g^{\beta,\gamma}$, so that $g_{i,j} g^{j,k} = \delta_i^k$. 
We freely raise and lower indices by contracting with the metric, denoting different covariant and contravariant versions of a tensor in the same manner. 
To emphasize the dependence or independence on the metric, we may use $\nabla f = \nabla_\alpha f$ for the (covariant) covector $d f$, while using $\nabla_g f = \nabla^\alpha f = g^{\alpha , \beta} \nabla_\beta f$ for the Riemannian (contravariant) gradient vector. 
We will frequently use $\scalar{\cdot,\cdot}$ to denote $g_{\alpha,\beta}$ and $\scalar{\cdot,\cdot}_g$ to denote $g^{\alpha,\beta}$; their diagonals will be denoted by $\abs{\cdot}^2$ and $\abs{\cdot}^2_g$, respectively.  
Similarly, $g_\alpha^\beta$, which is nothing but a pairing via the Kronecker delta  $\delta_{\alpha}^\beta$, will also be denoted by $\scalar{\cdot,\cdot}$. Our notation is aimed at taking note of the dependence on the metric $g$, while avoiding being too excessive, and should not create any ambiguity. For instance,
\begin{align*}
\abs{\nabla_g f}^2 & = \scalar{\nabla_g f,\nabla_g f} = g( \nabla_g f, \nabla_g f) = g_{i,j} \nabla^i f \nabla^j f \\
& = g_i^j \nabla^i f \nabla_j f = \scalar{g^{-1} \nabla f, \nabla f} = \scalar{\nabla_g f , \nabla f} \\
&= g^{i,j} \nabla_i f \nabla_j f = g^{-1}(\nabla f , \nabla f) 
 = \scalar{\nabla f , \nabla f}_g = \abs{\nabla f}_g^2 ,
\end{align*}
and
\begin{align*}
 \scalar{\Ric_{g,\mu}^{-1} \nabla_g f , \nabla_g f} & = g_{i,j} (\Ric_{g,\mu}^{-1})^i_k \nabla^k f \nabla^j f  \\
 & = (\Ric_{g,\mu}^{-1})^{i,j} \nabla_i f \nabla_j f = \Ric_{g,\mu}^{-1}(\nabla f,\nabla f) = \scalar{\Ric_{g,\mu}^{-1} \nabla f,\nabla f} \\
&  = g^{i,j} (\Ric_{g,\mu}^{-1})_i^k \nabla_k f \nabla_j f  = \scalar{\Ric_{g,\mu}^{-1} \; \nabla f , \nabla f}_g  .
 \end{align*}
 
When $M \subset \Real^d$, the symbols $\nabla$, $D^2$, $\det$, etc.\ are understood as the gradient, Hessian, determinant, etc. with respect to the standard Euclidean connection. The same operations with respect to the metric $g$ will be denoted with a subscript $g$: $\nabla_g, D^2_g$, etc.
Throughout this work, we identify between $\Real^d$ and its tangent spaces $T_x \Real^d$. We use $f_{x_i}$ to denote the partial derivative of $f$ with respect to $x_i$. A measure $\mu$ supported on $M \subset \Real^d$ is called log-concave if $M$ is convex and on $M$ we have $\mu = \exp(-V) dx$ with $V : M \rightarrow \Real$ convex. In the one-dimensional case $M \subset \Real$, all tensors are simply identified with scalar functions. 

\subsection{Bakry--\'Emery Criterion}

We will frequently employ the following classical criterion for the validity of the log-Sobolev inequality due to D. Bakry and M. \'Emery \cite{BakryEmery,BGL-Book}. In the case of a locally-convex boundary, this follows from the work of Qian \cite{QianGradientEstimateWithBoundary}, and in the case that the manifold is only assumed geodesically convex, this follows by a reduction to the one-dimensional case using the recent extension of the localization method to the Riemannian setting due to B. Klartag \cite{KlartagLocalizationOnManifolds} (the details are very similar to the ones appearing in the proof of Theorem \ref{thm:BL} below, and are omitted).

\begin{thm}[Bakry--\'Emery log-Sobolev Criterion] \label{thm:BE}
Let $(M,g,\mu)$ denote a weighted Riemannian manifold satisfying the Geometric Convexity Assumptions. 
If
\[
\Ric_{g,\mu} \geq \rho g ,
\]
then for all $f \in C^1(M)$ we have
\[
\Ent_\mu (f^2) \leq \frac{2}{\rho} \int \abs{\nabla_g f}^2 d\mu .
\]
\end{thm}

\noindent
Recall that the entropy functional $\Ent_\mu(h)$ of a non-negative measurable function $h$ with respect to the probability measure $\mu$ is defined as
\[
\Ent_{\mu}(h) := \int h \log h \; d\mu - \int h \; d\mu \cdot \log \int h \; d\mu .
\]

\subsection{Generalized Brascamp--Lieb Inequality}

The following theorem is a generalized version of the classical Brascamp--Lieb inequality from the Euclidean setting (\ref{eq:classical-BL}), which will play a fundamental role in this work. Modulo the technical assumptions hidden in the Geometric Convexity Assumptions, this generalized version on weighted manifolds is well-known to experts. Under assumption (1),  if the manifold is in addition assumed compact, this result was proved e.g. in \cite{KolesnikovEMilman-Reilly}; under assumption (2), the assertion follows as before from a reduction to the one-dimensional case following Klartag \cite{KlartagLocalizationOnManifolds}. We defer the details of the proof to the Appendix.

\begin{thm}[Generalized Brascamp--Lieb Inequality] \label{thm:BL}
Let $(M,g,\mu)$ denote a weighted Riemannian manifold satisfying the Geometric Convexity Assumptions. Assume that $\Ric_{g,\mu} > 0$ on $M$. 
Then for all $f \in C^1(M)$,
\[
\Var_\mu(f) \leq \int_M \scalar{\Ric_{g,\mu}^{-1} \nabla f,\nabla f} d\mu .
\]
\end{thm}

As an immediate corollary, which we state to emphasize the view-point employed throughout this entire work, we obtain the following way to upper bound the variance of a given function:

\begin{cor} \label{BLRiemann}
Given a smooth differentiable manifold $M$ endowed with an absolutely continuous probability measure $\mu$ having smooth positive density, we have
for all $f \in C^1(M)$,
\[
\Var_{\mu}(f) \leq \inf_g \set{ \int_M \scalar{ \Ric_{g,\mu}^{-1} \nabla f,\nabla f} d\mu } ~,
\]
where the infimum is taken over all Riemannian metrics $g$ on $M$ so that the assumptions of Theorem \ref{thm:BL} are satisfied. 
\end{cor}
\noindent
Note that generically, an optimal metric above will depend on the particular function $f$ one is testing. 

\subsection{Hessian Metrics}

A substantial part of this work is devoted to Hessian metrics, when $g$ is of the form
\[
g = D^2 \Phi ~,
\]
where $\Phi$ is a strongly convex smooth function on $M \subset \mathbb{R}^d$. Recall that a twice differentiable function is called strongly convex if its hessian is strictly (but not necessarily uniformly) positive-definite on its domain. 
Let $\mu = \exp(-V) dx$ denote a probability measure on $M$ having smooth density, and denote by $\nu$  the push-forward of $\mu$ by $\nabla \Phi$. Recall that a map $T$ is said to push-forward $\mu$ onto $\nu$ if
\begin{equation} \label{eq:push-forward}
\nu = T_*(\mu) = \mu \circ T^{-1} .
\end{equation}
Note that $\nabla \Phi$ is a smooth diffeomorphism thanks to the strong convexity of $\Phi$. The resulting change-of-variables formula
\begin{equation} \label{eq:change-of-variables}
\det  D^2 \Phi = \frac{\exp(-V)}{\exp(-W(\nabla \Phi))} 
\end{equation}
is known as the Monge-Amp\`ere equation, confirming in particular
that $\nu = \exp(-W) dx$ has a smooth density on $\nabla \Phi(\intr(M))$. It will be convenient to rewrite (\ref{eq:change-of-variables}) as
\begin{equation} \label{eq:log-change}
W(\nabla \Phi) = V + \log \det  D^2 \Phi  ~.
\end{equation}
Defining $P$ by $\exp(-V) dx = \exp(-P) \vol_g$, since
\[
 \vol_g = (\det D^2 \Phi)^{1/2} dx,
\]
it follows that
\[
P = V + \frac{1}{2} \log \det D^2 \Phi = \frac{1}{2} (V + W(\nabla \Phi)) ~.
\]
Recall that $\Ric_{g,\mu} = \Ric_{g} + D^2_g P$. The following calculation was verified by the first named author in \cite{Kolesnikov-HessianMetricsArisingFromOT}:

\begin{thm} \label{thm:Kol-compute}
The generalized Ricci-curvature tensor $\Ric_{g,\mu}$ of the weighted Riemannian manifold $(M,g = D^2 \Phi , \mu)$ is given by
\begin{equation} \label{be}
\Ric_{g,\mu} = \frac{1}{4} H + \frac{1}{2} \Bigl(  D^2 V+  D^2 \Phi \cdot  D^2 W (\nabla \Phi) \cdot   D^2 \Phi \Bigr) ~,
\end{equation}
where  $H$ is the non-negative symmetric matrix with entries
\[
H_{ij} = \mbox{\rm{Tr}} \Bigl[ (D^2 \Phi)^{-1} D^2 \Phi_{x_i}  (D^2 \Phi)^{-1} D^2 \Phi_{x_j} \Bigr] ~.
\]
\end{thm}

As apparent above, Hessian metrics have the advantage that (\ref{be}) only involves derivatives of $\Phi$ up to third order, and that if we neglect the non-negative $H$ term, only two orders are in fact relevant. We already see from (\ref{be}) that when $V$ and $W$ are both convex, then $\Ric_{g,\mu}$ is non-negative. Instead of neglecting the $H$ term completely, we may also estimate it from below using only two derivatives as follows:
\begin{lem} \label{lem:H}
As positive-definite matrices,
\begin{align*}
H & \geq \frac{1}{d} \nabla \log \det D^2 \Phi \otimes \nabla \log \det D^2 \Phi  \\
&= \frac{1}{d} \bigl( \nabla V - D^2 \Phi \cdot \nabla W(\nabla \Phi) \bigr) \otimes \bigl( \nabla V - D^2 \Phi \cdot \nabla W(\nabla \Phi)\bigr).
\end{align*}
\end{lem}
\begin{proof}
Given $\theta \in \Real^d$, introduce the following symmetric matrix:
\[
B_\theta = (D^2 \Phi)^{-1/2} D^2 \Phi_{x_\theta} (D^2 \Phi)^{-1/2},
\]
where $x_\theta = \sum_{i=1}^d \theta_i x_i$. Consequently, by the Cauchy--Schwarz inequality applied to the eigenvalues of $B_\theta$, we have
\[
\scalar{H \theta, \theta} =   \mbox{\rm{Tr}}(B_\theta^2) \geq \frac{1}{d} \mbox{\rm{Tr}}(B_\theta)^2 = \frac{1}{d} \brac{ \mbox{\rm{Tr}} \Bigl[ (D^2 \Phi)^{-1} D^2 \Phi_{x_\theta} \Bigr] } ^2 = \frac{1}{d} \scalar{\nabla \log \det D^2 \Phi, \theta}^2 ,
\]
and hence
\[
 H \ge \frac{1}{d} \nabla \log \det D^2 \Phi \otimes \nabla \log \det D^2 \Phi . 
\]
The second part of the assertion follows by differentiating (\ref{eq:log-change}). 
\end{proof}

\subsection{Optimal Transport}

In the previous subsection, we have described a situation where given $\mu = \exp(-V) dx$ and a strongly convex $\Phi$, the measure $\nu = \exp(-W) dx$ is constructed as the push-forward of $\mu$ via $\nabla \Phi$. From this perspective, $\nu$ and $W$ are auxiliary objects, useful mainly for computational reasons. Remarkably, it was shown by Y. Brenier \cite{BrenierMap} (see also McCann \cite{McCannConvexityPrincipleForGases} for refinements), that it is possible to reverse the above roles: starting with two absolutely-continuous probability measures $\mu = \exp(-V) dx$ and $\nu = \exp(-W) dx$ (supported on $\supp(\mu), \supp(\nu) \subset \Real^d$, respectively), there exists  a \emph{strictly} convex $\Phi$ so that $T: x \mapsto \nabla \Phi(x)$ pushes forward $\mu$ onto $\nu$. In fact, such a $T$ is unique $\mu$-almost-everywhere, and is characterized as being the $L^2$-optimal-transport map between $\mu$ and $\nu$, minimizing the transport-cost $\int \abs{T(x) - x}^2 d\mu(x)$ among all maps satisfying (\ref{eq:push-forward}) - see \cite{BogachevKolesnikov-OTSurvey, VillaniTopicsInOptimalTransport} for a detailed discussion. 

We will rely on the known regularity theory for the Monge-Amp\`ere equation (\ref{eq:change-of-variables}) associated to the above transport problem. It was shown by Caffarelli \cite{CaffarelliStrictlyConvexIsHolder,CaffarelliRegularity,CaffarelliHigherHolderRegularity} that whenever the support of $\nu$ is convex, the usual interior elliptic regularity estimates hold: if $V \in C^{k,\alpha}_{loc}(\intr(\supp(\mu)))$ and $
W \in C^{k,\alpha}_{loc}(\intr(\supp(\nu)))$, then $\Phi \in C^{k+2,\alpha}_{loc}(\intr(\supp(\mu)))$, for all $k \geq 0$. In particular, the change-of-variables formula (\ref{eq:change-of-variables}) applies in the interior of $M$, and $\Phi$ is in fact strongly convex there.

\section{Refined Brascamp--Lieb inequalities} \label{sec:refined-BL}

At least on a formal level, it is clear that Corollary \ref{BLRiemann}, together with the computations of Theorem \ref{thm:Kol-compute} and Lemma \ref{lem:H}, together imply the following:
\begin{thm}[Refined Brascamp--Lieb on $\Real^d$] \label{thm:refined-BL}
Let $\mu = \exp(-V) dx$ denote a probability measure having positive smooth density on $\Real^d$ so that $D^2 V > 0$. 
Then for any $f \in C^1(\Real^d)$,
\[
\Var_{\mu}(f) \leq 2 \; \inf_W \set{ \int_{\Real^d} \scalar{(D^2 V + Q_W + Q_H)^{-1} \nabla f,\nabla f} d\mu } ~,
\]
where:
\begin{align*} 
\nonumber
Q_W & :=  D^2 \Phi \cdot D^2 W(\nabla \Phi) \cdot D^2 \Phi ,\\
 Q_H & := \frac{1}{2d} \bigl( \nabla V - D^2 \Phi \cdot \nabla W(\nabla \Phi) \bigr) \otimes \bigl( \nabla V - D^2 \Phi \cdot \nabla W(\nabla \Phi)\bigr),
\end{align*}
the infimum is taken over all smooth convex functions $W$ on a convex subset $\Omega \subset \Real^d$, so that $\nu = \exp(-W) dx$ is a probability measure on $\Omega$, and $\nabla \Phi$ denotes the (smooth) optimal-transport map pushing forward $\mu$ onto $\nu$.
\end{thm}

We stress that further work is required to justify the latter theorem, since we do not know whether the associated weighted manifold $(\Real^d,g = D^2 \Phi,\mu)$ will in general satisfy the Geometric Convexity Assumptions required to apply Theorem \ref{thm:BL}. This delicate point is circumvented in our proof of Theorem \ref{thm:appendix-refinedBL}, whose formulation slightly generalizes that of Theorem \ref{thm:refined-BL} above, and is deferred to the Appendix. 

\medskip

Note that by using $W = V$ (and hence $\nabla \Phi = \mbox{\rm{Id}}$) above, we recover the classical Brascamp-Lieb inequality (\ref{eq:classical-BL}) on $\Real^d$, so Theorem \ref{thm:refined-BL} should be understood as a refinement. Moreover, since
\[
 D^2 V + Q_W + Q_H \geq D^2 V + Q_W \geq D^2 V 
\]
whenever $D^2 W \geq 0$, we see that \emph{any} convex $W$ yields an estimate which is, up to a factor of $2$, at least as good as the Brascamp-Lieb one. We will see in Corollary \ref{cor:finite-dim-BL} below how to also exploit the additional $Q_H$ term appearing above.

\medskip

Let us now demonstrate the usefulness of the Refined Brascamp--Lieb Theorem \ref{thm:refined-BL} in two concrete examples. 

\subsection{$(-d)$-dimensional Brascamp--Lieb inequality} 

First, we set $W$ to be constant on any convex compact set $\Omega$ with non-empty interior (say, the Euclidean ball). 

\begin{cor} \label{cor:finite-dim-BL}
Let $\mu = \exp(-V) dx$ denote a log-concave probability measure having positive smooth density on $\Real^d$. 
Then for any $f \in C^1(\Real^d)$,
\[
\Var_{\mu}(f) \le C_{d} \int \big\langle \big(D^2 V + \frac{1}{2d} \nabla V \otimes \nabla V \big)^{-1} \nabla f, \nabla f \big \rangle \ d \mu.
\]
with $C_d = 2$.  \end{cor}
\begin{proof}
Assume without loss of generality that $f$ is supported on a closed Euclidean ball $B_R$. Applying Theorem \ref{thm:refined-BL} with $W$ as above to strongly convex functions $V_n$ approximating $V$ in $C^2_{loc}(\Real^d)$, the result follows by passing to the limit and using that $V_n \to V$ in $C^2(B_R)$. 
\end{proof}

As explained e.g. in \cite{KolesnikovEMilman-Reilly}, the tensor
\[
D^2 V + \frac{1}{2d} \nabla V \otimes \nabla V 
\]
may be interpreted as a $(-d)$-dimensional generalized Ricci tensor (see Section \ref{sec:conformal} for the general definition), so the above should be thought of as a $(-d)$-dimensional Brascamp--Lieb inequality. In the Euclidean setting, such an inequality was first proved by Bobkov and Ledoux \cite{BobkovLedouxWeightedPoincareForHeavyTails}, and subsequently generalized and sharpened by V.-H. Nguyen \cite{NguyenDimensionalBrascampLieb}, who obtained a better constant of $C_d = \frac{2d}{2d-1} \leq 2$, which is best possible in the Euclidean setting. On general (compact) weighted manifolds, a suitable generalization was obtained in \cite{KolesnikovEMilman-Reilly} with $C_d = \frac{d+1}{d}$, and it was shown that the latter constant, while inferior to Nguyen's Euclidean constant, is in general best possible on weighted manifolds.  

\subsection{Brascamp--Lieb inequality for compactly supported measures}

\begin{thm}[Brascamp--Lieb with compact support]
Let $\nu = \exp(-W) dx$ denote a log-concave probability measure with barycenter at the origin, having smooth positive density on its convex support $\Omega \subset \{x:|x| \le R\}$. 
Then for any $h \in C^1(\Omega)$,
\begin{equation} \label{eq:BL-compact}
\Var_{\nu}(h) \le 2 \int \big\langle \big( \frac{1}{2 R^2} \cdot \mbox{{\rm Id}}+ D^2 W \big)^{-1} \nabla h, \nabla h   \big \rangle \ d \nu ,
\end{equation}
where $C>0$ is a universal constant. 
\end{thm}
\begin{proof}
We would like to apply the Refined Brascamp--Lieb Theorem \ref{thm:refined-BL} with $V = \Phi$, i.e. obtain a convex fixed-point solution $V = \Phi$ to the Monge-Amp\`ere equation (\ref{eq:change-of-variables}). The resulting non-linear elliptic PDE
\begin{equation} \label{eq:Kahler-change}
\exp(-\Phi) = \det D^2 \Phi \cdot \exp(-W(\nabla \Phi)) 
\end{equation}
is referred to (at least in the complex setting) as the K{\"a}hler--Einstein equation.
Under the necessary assumption that the barycenter of $\nu$ is at the origin, the existence (and uniqueness up to translation) of such a solution was proved by Berman and Berndtsson \cite{BermanBerndtsson-RealMongeAmpere} (see also Wang and Zhu \cite{WangZhu-KahlerEinstein} for the analogous result in the complex setting and relation to Kahler-Einstein metrics, and the work of D. Cordero--Erausquin and B. Klartag \cite{CorderoKlartag-MomentMeasures}, where existence and uniqueness of weak solutions to (\ref{eq:Kahler-change}) for general Borel measures $\nu$ was precisely characterized). As verified in \cite{BermanBerndtsson-RealMongeAmpere}, the usual regularity theory for the Monge--Amp\`ere equation implies that $\Phi$ is smooth on $\Real^d$, so that (\ref{eq:Kahler-change}) indeed holds in the classical sense, yielding in particular that $\Phi$ is strongly convex. Let $\mu  = \exp(-\Phi) dx$ denote the corresponding log-concave probability measure. 

Now given $h \in C^1(\Omega)$, let us apply the Refined Brascamp--Lieb Theorem \ref{thm:refined-BL} to bound the variance of $f = h(\nabla \Phi)$ with respect to $\mu$:
\begin{align*}
\Var_ {\nu}(h) & = \Var_ {\mu}(f) \le 
 2  \int \scalar{  \bigl(  D^2 \Phi + D^2 \Phi \cdot D^2 W(\nabla \Phi) \cdot D^2 \Phi \bigr)^{-1}  \ \nabla f, \nabla f } d \mu
 \\& = 2  \int \scalar{ D^2 \Phi \cdot \Bigl(  D^2 \Phi + D^2 \Phi \cdot D^2 W(\nabla \Phi) \cdot D^2 \Phi \Bigr)^{-1}  D^2 \Phi \ \nabla h(\nabla \Phi), \nabla h(\nabla \Phi) } d \mu \\
 & = 2 \int \scalar{ \brac{ (D^2 \Phi)^{-1} + D^2 W(\nabla \Phi) }^{-1}   \ \nabla h(\nabla \Phi), \nabla h(\nabla \Phi) } d \mu .
\end{align*}
Applying the estimate
\[
\text{tr} (D^2 \Phi )\le   2 R^2
\]
obtained by Klartag in \cite{Klartag-PayneWeinbergerViaMomentMeasures}, so that in particular
\[
( \; 0 \leq \; ) \; D^2 \Phi \leq 2 R^2 \; \rm{Id} ,
\]
 we obtain
\begin{align*}
\Var_ {\nu}(h) & \le 2  \int \scalar{ \Bigl(  \frac{1}{2 R^2} \mbox{\rm{Id}} +  D^2 W(\nabla \Phi)  \Bigr)^{-1} \nabla h(\nabla \Phi), \nabla h(\nabla \Phi) } d \mu
\\&  =  2  \int \scalar{ \Bigl(  \frac{1}{2 R^2} \mbox{\rm{Id}} +  D^2 W  \Bigr)^{-1} \nabla h, \nabla h } d \nu ,
\end{align*}
as asserted. 
\end{proof}

\begin{rem}
The above argument was used by Klartag in \cite{Klartag-PayneWeinbergerViaMomentMeasures} to deduce the following Payne--Weinberger-type estimate:
\begin{equation} \label{eq:PW}
\Var_{\nu}(h) \le 2 R^2 \int \abs{\nabla h}^2 d\nu .
\end{equation}
Our improvement over this estimate is due to our usage of the Refined Brascamp--Lieb Theorem \ref{thm:refined-BL}, instead of the classical Brascamp--Lieb inequality (\ref{eq:classical-BL}) employed by Klartag (in one of his alternative derivations). Note that our estimate (\ref{eq:BL-compact}) nicely fuses between the Payne-Weinberger estimate (\ref{eq:PW}) and the Brascamp--Lieb one (\ref{eq:classical-BL}). 
\end{rem}

\section{Entropic version of the Brascamp-Lieb inequality} \label{sec:entropic-BL}

Let $\mu = \exp(-V) dx$  denote a log-concave probability measure supported on a convex set $\Omega \subset \Real^d$ having smooth (possibly empty) boundary, and assume that the potential $V$ is strongly convex and smooth in $\Omega$. Let us now push-forward $\mu$ onto the measure $\nu = \exp(-W) dx$ using the mapping $T(x) = \nabla \Phi(x)$ with $\Phi  =  V$. We equip $\Omega$ with the Riemannian metric $g = D^2 \Phi = D^2 V$.  

\medskip

According to the Bakry--\'Emery criterion (Theorem \ref{thm:BE}), the entropic version
\begin{equation} \label{eq:entropic-BL}
\Ent_{\mu} (f^2) \le \frac{2}{\rho} \int  \scalar{ (D^2 V)^{-1} \nabla f, \nabla f } \ d \mu
\end{equation}
of the classical  Brascamp--Lieb inequality
\begin{equation} \label{eq:BL-again}
\Var_{\mu} (f) \le \int \langle (D^2 V)^{-1} \nabla f, \nabla f \rangle \ d \mu ,
\end{equation}
holds for all $f \in C^1(\Omega)$, provided the Geometric Convexity Assumptions are satisfied on $(\Omega,g = D^2 V,\mu)$ and
\[
\Ric_{g,\mu} \ge \rho \cdot D^2 V \text{ on $\Omega$ }.
\] 
 
In this section, we would like to obtain a tractable condition for verifying the validity of (\ref{eq:entropic-BL}). 
Note that while any log-concave measure satisfies (\ref{eq:BL-again}), it is known (see Bobkov--Ledoux \cite{BobkovLedoux}) that (\ref{eq:entropic-BL}) \emph{cannot} hold (with any $\rho > 0$) for arbitrary log-concave measures. 

\subsection{Expressing everything solely as a function of $V$}

We start by expressing the generalized Ricci tensor $\Ric_{g,\mu}$ as a function of $V$ only.  Let $V^* : \Real^n \rightarrow \Real \cup \set{+\infty}$ denote
the Legendre transform of $V$ (see e.g. \cite[Section 12]{RockafellarBook}), given by
\begin{equation} \label{eq:Legendre}
V^*(y) = \sup_{x \in \Omega} \bigl( \scalar{x,y} - V(x) \bigr).
\end{equation}
It will be sufficient for our purposes to restrict $V^*$ onto the domain
\[
\Omega^* = \intr (\nabla V(\Omega)) ,
\]
where it is necessarily finite. 
Smoothness and strong convexity of $V$ on $\Omega$  imply that given $y \in \Omega^*$, the supremum in (\ref{eq:Legendre}) is attained at the unique point $x = \nabla V^*(y)$ where $y = \nabla V(x)$, and that $V^*$ is smooth and strongly convex on $\Omega^*$ \cite{RockafellarBook}. The following well-known identities then easily follow:
\[
 V(x) + V^*(\nabla V(x)) = \scalar{ x, \nabla V(x) } ~ , ~ \nabla V^*(\nabla V(x)) = x ~,~ \ D^2 V^*(\nabla V) \cdot D^2 V = \mbox{Id} ,
\]
for every $x$ in the interior of $\Omega$.
Recalling the change-of-variables formula (\ref{eq:log-change}):
\[
W(\nabla V) = V + \log \det D^2 V ,
\]
we set
\[
F = W +  V^*.
\]
Using the above properties of the Legendre transform, we see that
\[
F(\nabla V(x)) = \scalar{ x, \nabla V(x) } + \log \det D^2 V(x) ,
\]
and hence
\begin{equation} \label{eq:F}
F(y) =  \scalar{ y, \nabla V^*(y) } - \log \det D^2 V^*(y) ~,~ \ y \in \Omega^* . 
\end{equation}
Denoting as usual
\begin{equation} \label{eq:H}
H_{ij} = \mbox{{\rm Tr}} \Bigl[ (D^2 V)^{-1} D^2 V_{x_i}  (D^2 V)^{-1} D^2 V_{x_j} \Bigr] ,
\end{equation}
it is then easy to check that Theorem \ref{thm:Kol-compute} yields the following:

\begin{prop} \label{prop:PhiV-Compute}
Let $\mu = \exp(-V) dx$  denote a log-concave probability measure supported on a convex domain $\Omega \subset \Real^d$, and assume that the potential $V$ is strongly convex and smooth in $\Omega$. Then the weighted manifold
\[
 (\Omega,g = D^2 V, \mu = \exp(-V) dx)
\]
satisfies
\begin{equation} \label{eq:entropic-Ric}
\mbox{\rm{Ric}}_{g,\mu} = \frac{1}{4} H + \frac{1}{2} D^2 V \cdot  D^2 F (\nabla V) \cdot   D^2 V,
\end{equation}
with $F$ and $H$ defined by (\ref{eq:F}) and (\ref{eq:H}). 
\end{prop}

Recall by Lemma \ref{lem:H} that
\[
 H \ge \frac{1}{d} \nabla \log \det D^2 V \otimes \nabla \log \det D^2 V ,
\]
which by the properties of the Legendre transform is equivalent to
\[
 D^2 V^* \cdot H(\nabla V^*) \cdot D^2 V^* \ge \frac{1}{d} \nabla \log \det D^2 V^* \otimes \nabla \log \det D^2 V^*.
\]
Multiplying (\ref{eq:entropic-Ric}) by $D^2 V^*$ from both sides and evaluating at $\nabla V^*$, we obtain from (\ref{eq:entropic-BL}) the following:

\begin{thm} \label{thm:entropic-BL}
With the same assumptions as in Proposition \ref{prop:PhiV-Compute}, assume in addition that
\[
D^2 F + \frac{1}{2}   D^2 V^* \cdot H(\nabla V^*) \cdot D^2 V^* \ge 2 \rho(\nabla V^*) \cdot D^2 V^* ,
\]
for some function $\rho : \Omega \rightarrow \Real$. Then on $(\Omega,g = D^2 V, \mu = \exp(-V) dx)$ we have
\[
\Ric_{g,\mu}(x) \ge \rho(x) \cdot g(x) .
\]
In particular, if
\[
 D^2  F +  \frac{1}{2d} \nabla \log \det D^2 V^* \otimes \nabla \log \det D^2 V^*  \ge 2 \rho \cdot D^2 V^* 
\]
for some $\rho>0$, and $(\Omega,g,\mu)$ satisfies the Geometric Convexity Assumptions, then the following entropic Brascamp-Lieb inequality holds:
\[
\Ent_\mu(f^2) \le \frac{2}{\rho} \int \langle (D^2 V)^{-1} \nabla f, \nabla f \rangle \ d \mu \;\;\; \forall f \in C^1(\Omega) . 
\]
\end{thm}

\begin{rem}
In \cite{BobkovLedoux}, Bobkov and Ledoux obtained a result in the same direction, which however does not seem to be directly comparable to Theorem \ref{thm:entropic-BL}. Namely, they showed that
\[
\Ent_\mu (f^2) \le 3 \int  \langle (D^2 V)^{-1} \nabla f, \nabla f \rangle \ d \mu \;\;\; \forall f \in C^1(\Omega) ,
\]
provided $V$ is convex and $x \to V_{hh}(x)$ is concave for every $h$.\\
Note that in our formulation, it is enough to check convexity of the single function $F - 2 \rho V^*$, unlike in the Bobkov--Ledoux result.
\end{rem}

\subsection{One-Dimensional Case}

In the one-dimensional case $\Omega \subset \Real$, we can use either Theorem \ref{thm:entropic-BL} or apply the following  exact expression for the space $(\Omega, g = V''(x) (dx)^2 , \exp(-V) dx)$: 
\[
\Ric_{g,\mu} =  V'' + \frac{1}{2} \frac{V^{(4)}}{V''} -\frac{3}{4} \Bigl( \frac{V'''}{V''}\Bigr)^2
- \frac{1}{2} \frac{V' V'''}{V''} 
\]
(recall that all tensors are identified in the one-dimensional case with scalar functions). 
Since the space is one-dimensional, the Ricci part of the tensor vanishes and in fact the right-hand side equals to the
Hessian of the  corresponding potential.
\begin{example}
Consider the following probability measure on $\Real_+$:
\[
\mu_q = \frac{1}{Z_{q,c}}  \exp(-c x^q) dx ,
\]
with $1 < q \leq 2$ and $c > 0$. The manifold \[
((0,\infty) , g = c  q(q-1) x^{q-2} (dx)^2)
\]
is clearly isometric to $(0,\infty)$ with the standard Euclidean metric, and $((0,\infty),g,\mu_q)$ satisfies the Geometric Convexity Assumptions. A calculation yields
\begin{equation}
\label{ric-q}
\Ric_{g,\mu_q} =  \frac{cq^2}{2} x^{q-2} + \frac{q(2-q)}{4} x^{-2} \geq  \frac{q}{2(q-1)} g. 
\end{equation}
This implies the following entropic Brascamp-Lieb inequality:
\[
\Ent_{\mu_q}(f^2) \le \frac{4}{cq^2} \int  x^{2-q} (f')^2 \ d \mu_q \;\;\; \forall f \in C^1(\Real_+) ,
\]
which by the tensorization property of the log-Sobolev inequality \cite[Corollary 5.7]{Ledoux-Book} yields the following entropic Brascamp-Lieb inequality for the product measure:
\begin{equation}
\label{exppower}
\Ent_{\mu_q^{\otimes d}}(f^2) \le \frac{4}{cq^2} \int \sum_{i=1}^d x^{2-q}_i  f_{x_i}^2 \ d \mu_q^{\otimes d}  \;\;\; \forall f \in C^1(\Real_+^d) .
\end{equation}
By taking the limit $q \rightarrow 1$, the above in fact applies to the entire range $q \in [1,2]$.
Applying the change of variables $t_i = x^q_i$, we obtain the following inequality due to D.~Bakry \cite{Bakry-JacobiSemigroups}:
\[
\Ent_{\nu}(f^2) \le \frac{4}{cq} \int\sum_{i=1}^{d} t^{\frac{1}{q}}_i   f_{t_i}^2 d \nu \;\;\; \forall q \in [1,2] \;\;\; \forall f \in C^1(\Real_+^d) ,
\]
where $\nu = \exp(-c \sum_{i=1}^d t_i) dt$ is the exponential measure on $\Real_+^d$. 
\end{example}

\begin{example}
It is natural to ask whether it is possible to obtain an analogue of (\ref{exppower}) for $q > 2$. It is easy to check that for $V(x)=x^q$ and $q>2$, 
the corresponding generalized Ricci tensor is negative for small values of $x$ (see (\ref{ric-q})). To overcome this difficulty, let us consider an even function 
$V$ on $\mathbb{R}$ which is quadratic for small values of $\abs{x}$ and behaves like $\abs{x}^q$ for large values of $\abs{x}$. More precisely, let
$V^*$, the Legendre transform of $V$, be defined by
\[
(V^*)''(y) = \min\Bigl\{\frac{1}{p}, \frac{1}{p} \abs{y}^{p-2}\Bigr\}\;\; , \;\; V^*(0) = (V^*)'(0) = 0 \;\; , \;\; p=\frac{q}{q-1}.
\]
A direct computation (involving our normalization with $\frac{1}{p}$ above) verifies that
\[
F(y) = y (V^*)'(y) - \log (V^*)''(y)
\]
is convex. Moreover, $F - 2 \rho V^{*}$ remains convex for small values of $\rho \in (0, \rho_q]$. Applying Theorem 
\ref{thm:entropic-BL}, we deduce that for any $q >2$,
\[
\Ent_{\mu}(f^2) \le \frac{2}{\rho_q} \int  \min(1, x^{2-q})  (f')^2  d \mu \;\;\; \forall f \in C^1(\Real_+) ,
\]
where $\mu = \frac{1}{Z} \exp(-V) dx$ on $\Real_+$.  
Note that $V$ is quadratic for small values of $x$, and that there exist constants $a_q,b_q > 0$ so that
\[
\abs{ V(x)  - a_q x^q } \leq b_q 
\]
 for large (and consequently all) values of $x \in \Real_+$. It follows by the Holley--Stroock perturbation lemma (e.g. \cite[Proposition 5.5]{Ledoux-Book}) and rescaling that
\[
\Ent_{\mu_q} (f^2) \le C_q \int  \min(1, x^{2-q})  (f')^2  d \mu_q \;\;\; \forall f \in C^1(\Real_+),
\]
where $\mu_q = \frac{1}{Z_q} \exp(-x^q) dx$ on $\Real_+$. By tensorization, we finally obtain
\[
\Ent_{\mu_q^{\otimes d}} (f^2) \le C_q \int \sum_{i=1}^d \min(1, x^{2-q}_i)  f_{x_i}^2 \ d \mu_q^{\otimes d} \;\;\; \forall q \in [2,\infty) \;\;\; \forall f \in C^1(\Real_+^d) . 
\]
\end{example}

The above examples of product measures should be compared to the more general analysis in the next section, in which product metrics other than $D^2 V$ will be considered.

\section{Product Metrics and Unconditional Convex Sets} \label{sec:unconditional}
 
In this section, we investigate the consequences of choosing a particularly simple metric $g = D^2 \Phi$ -- a \emph{product} metric. 
To this end, recall that by Theorem \ref{thm:Kol-compute}, the formula for the generalized Ricci tensor on $(\Real^n, g = D^2 \Phi , \mu = \exp(-V) dx)$ is
\[
\Ric_{g,\mu} = \frac{1}{4} H + \frac{1}{2} \Bigl(  D^2 V+  D^2 \Phi \cdot  D^2 W (\nabla \Phi) \cdot   D^2 \Phi \Bigr) ,
\]
and rewrite it as a function of $\Phi$. Differentiating twice the change of variables formula
\[
W(\nabla \Phi) = V + \log \det D^2 \Phi,
\]
we obtain:
\[
D^2 \Phi \cdot \nabla W(\nabla \Phi) = \nabla V + \nabla \log \det D^2 \Phi,
\]
\[
D^2 \Phi \cdot D^2 W(\nabla \Phi) \cdot D^2 \Phi + \sum_i  (D^2 \Phi)_{x_i} \cdot W_{x_i}(\nabla \Phi) 
= D^2 V + D^2 \log \det D^2 \Phi. 
\]
Collecting everything together, we deduce
\begin{equation} \label{BEVV}
\Ric_{g,\mu} =
D^2 V + \frac{1}{2} D^2 \log \det D^2 \Phi + \frac{1}{4} H -  \frac{1}{2}\sum_{i=1}^d (D^2 \Phi)_{x_i} \cdot
\langle (D^2 \Phi)^{-1} e_i , \nabla V  + \nabla  \log \det D^2 \Phi\rangle,
\end{equation}
where as usual
\[
H_{ij} = \mbox{\rm{Tr}} \Bigl[ (D^2 \Phi)^{-1} D^2 \Phi_{x_i}  (D^2 \Phi)^{-1} D^2 \Phi_{x_j} \Bigr] ~.
\]
A careful computation in the case of a product metric then verifies the following:

\begin{prop} \label{prop:product-Ric}
Let $\Phi$ denote a smooth strongly convex function on $\Omega \subset \Real^d$ of the form
\[
\Phi(x) = \sum_{i=1}^d \Phi_{i}(x_i)  .
\]
Set
\[
 u_i  = 1 / \sqrt{\Phi''_{i}} .
\]
Then the generalized Ricci tensor on $(\Omega, g = D^2 \Phi , \mu = \exp(-V) dx)$ is given by
\[
\Ric_{g,\mu} = D^2 V + \mbox{\rm{diag}}\Bigl\{ V_{x_i}  \frac{u_i'(x_i)}{u_i(x_i)} - \frac{u''_i(x_i)}{u_i(x_i)} \Bigr\}.
\]
\end{prop}

\begin{rem} \label{rem:christoffel}
Since any one-dimensional metric is locally isometric to the standard Euclidean one, this also holds for product metrics, and our manifold is locally isometric to Euclidean. Consequently, the geometric Ricci tensor $\Ric_g$ vanishes in the above situation and hence $\Ric_{g,\mu} = {\mbox{\rm Hess}}_g P$ where $\mu = \exp(-P) d\vol_g$. In addition, this implies the vanishing of the Christoffel symbols $\Gamma^{i}_{jk}$, unless $i=j=k$. 
\end{rem}

We will apply the above formula for the study of measures on the principal orthant $(0,\infty)^d$. To this end, it will be useful to make the following:
\begin{defn}[Orthant Unconditional]
A convex set $\Omega \subset (0,\infty)^d$ is called orthant unconditional if every outer normal (with respect to the standard Euclidean structure) of $\partial \Omega \cap (0,\infty)^d$ has non-negative coordinates. Equivalently, a convex $\Omega$ is orthant unconditional if $\Omega = K \cap (0,\infty)^d$ where $K$ is an unconditional convex set, i.e. invariant under reflections with respect to the coordinate hyperplanes. Similarly, a measure $\mu$ on $\Real^d$ is called unconditional if it is invariant under the latter reflections. 
\end{defn}

Clearly $(0,\infty)^d$ is geodesically convex with respect to any product metric, as it is isometric to a product subset of Euclidean space (which may be bounded or unbounded, complete or not). The next lemma addresses the convexity of orthant unconditional convex subsets of $(0,\infty)^d$. 
\begin{lem} \label{lem:orthant-convex}
Let $\Omega \subset  (0,\infty)^d$ denote a relatively closed, orthant unconditional convex set, having smooth relative boundary $\partial_r \Omega = \partial \Omega \cap (0,\infty)^d$. Let $g(x) = \sum_{i=1}^d g_{i}(x_i)  (dx^i)^2$ denote a product metric on $(0,\infty)^d$. If
\begin{equation} \label{eq:g-decrease}
g_{i}' \leq 0 \;\;\; \forall i=1,\ldots,d ,
\end{equation}
then $\partial_r \Omega \subset (\Omega,g)$ is locally-convex, and moreover, both $(\Omega,g)$ and $(\intr(\Omega),g)$ are geodesically convex. 
\end{lem}
\begin{proof}
First, we claim that it suffices to show that any smooth convex function $F$ on $(0,\infty)^d$ with respect to the Euclidean metric, such that $F_{x_i} \ge 0$ for every $i=1,\ldots,d$, is also convex with respect to the metric $g$. The latter means that $\mbox{\rm{Hess}}_g (F) \geq 0$, or equivalently, that for any $g$-geodesic $t \mapsto \gamma(t)$ it holds that $t \mapsto F(\gamma(t))$ is convex in the usual sense. 

Indeed, this property implies the asserted convexity properties of $(\Omega,g)$ by inspecting the gauge function $p(x) := \min \set{ t > 0 ; x \in t \Omega}$ on $(0,\infty)^d$. Note that $p(x) \leq 1$ iff $x \in \Omega$ and $p(x) = 1$ iff $x \in \partial_r \Omega$. Furthermore, $p$ is clearly Euclidean convex and satisfies $p_{x_i} \ge 0$ thanks to the unconditionality property. Our ansatz would imply that $p$ is $g$-convex, and since $p|_{\partial_r \Omega} = 1$, it follows that the second fundamental form on $\partial_r \Omega$ is non-negative, yielding local convexity. The geodesic convexity is established similarly: a geodesic $\gamma : [0,1] \rightarrow (0,\infty)^d$ connecting two points in $\Omega$ ($\intr(\Omega)$) must lie entirely inside $\Omega$ ($\intr(\Omega)$), since $p(\gamma(t))$ is convex and $p(\gamma(0)), p(\gamma(1)) \leq 1 ( < 1)$. 

To establish our ansatz, recall by Remark \ref{rem:christoffel} that $\Gamma^{k}_{ij}=0$ unless $i=j=k$, and note by (\ref{eq:g-decrease})
 that
\[
\Gamma^k_{k,k} = \frac{1}{2} g^{k,k} \frac{\partial g_{k,k}}{\partial x^k} = \frac{1}{2} \frac{g_k'}{g_k} \leq 0 .
\]
Consequently,
\[
(\mbox{\rm{Hess}}_g)_{i,j} F = D^2_{i,j} F - \Gamma^k_{i,j} F_{x_k} =  \set{D^2  F - \mbox{\rm{diag}}(\Gamma^k_{k,k} F_{x_k})}_{i,j} ,
\]
and we see that $\mbox{Hess}_g F \geq 0$, as asserted. 
\end{proof}

\subsection{Polynomially decaying product metrics}

In this subsection, we specialize to the case $u_i(x) = u(x) = \abs{x}^p$, $p \in (0,1)$. 

\begin{thm} \label{thm:product-poly}
Let $\mu = \exp(-V) dx$ denote a probability measure with smooth and positive density, supported in $\Omega$, an orthant unconditional convex subset of $(0,\infty)^d$. Then the generalized Ricci tensor of the weighted manifold
\[
 \brac{\Omega  , g_p = \sum_{i=1}^d  x_i^{-2p} (dx^i)^2 , \mu = \exp(-V) dx}
\]
satisfies
\[
\Ric_{g_p,\mu} = D^2 V + \mbox{\rm{diag}}\Bigl\{ V_{x_i} \frac{p}{x_i}  + \frac{p (1-p)}{x^2_i} \Bigr\} .
\]
Consequently, for all $f \in C^1(\Real_+^d)$ and $p \in (0,1)$:
\begin{enumerate}
\item If $Ric_{g_p,\mu} > 0$ then
\[
\Var_{\mu}(f) \le  \int \scalar{ \Ric_{g_p,\mu}^{-1}  \nabla f, \nabla f } d \mu .
\]
\item If
\[
 D^2 V \geq 0 ~,~ V_{x_i} \geq 0 \;\;\; \forall i=1,\ldots d \;\; , 
 \]
 then
 \[
 \Var_{\mu}(f) \le  4 \int \sum_{i=1}^d x_i^2 f_{x_i}^2 d\mu .
 \]
 In particular, if $\Omega \subset [0,R]^d$ then
 \[
  \Var_{\mu}(f) \le  4  R^2 \int \abs{\nabla f}^2 d\mu  .
 \]
 \item
  If
\[
 D^2 V \geq 0 ~,~ V_{x_i} \geq \lambda \;\;\; \forall i=1,\ldots d \;\; , 
 \]
 then
 \[
 \Var_{\mu}(f) \le  \frac{1}{\lambda} \int \sum_{i=1}^d x_i  f_{x_i}^2 d\mu .
 \]
 \item If $\Omega \subset [0,R]^d$ and
 \[
 D^2 V \geq 0 ~,~ V_{x_i} \geq 0 \;\;\; \forall i=1,\ldots d \;\; , 
 \]
 then for every $p \in (0,1)$, $\Ric_{g_p,\mu} \geq \rho_p g_p$ with
 \[
 \rho_p := \frac{p (1-p)}{R^{2 - 2p}} ,
 \]
 and hence
 \begin{equation} \label{eq:diagonal-BE}
\Ent_{\mu}(f^2)  \leq \frac{2}{\rho_p} \int \sum_{i=1}^d x_i^{2p} f_{x_i}^2 d\mu .
 \end{equation}
  In particular,
 \[
 \Ent_{\mu}(f^2)  \leq 8 R^2 \int \abs{\nabla f}^2 d\mu .
 \]
\item If
\[
 D^2 V \geq 0 ~,~ V_{x_i} \geq \lambda > 0  \;\;\; \forall i=1,\ldots d \;\; ,
 \]
 then for any $p \in [1/2,1)$, $\Ric_{g_p,\mu} \geq \rho_p g_p$ with
 \[
 \rho_p := \brac{\frac{\lambda p}{2- 2p}}^{2 - 2p} \brac{\frac{ p (1-p)}{2p-1}}^{2p-1} ,  \]
 and (\ref{eq:diagonal-BE}) holds. In particular,
 \[
  \Ent_{\mu}(f^2)  \leq \frac{4}{\lambda} \int \sum_{i=1}^d x_i f_{x_i}^2 d\mu .
 \]
 \end{enumerate}

\end{thm}

\begin{proof}
Note that when $p \in (0,1)$, since $x^{-p}$ is integrable at the origin but not at infinity, then $[0,\infty)$ equipped with the metric $x^{-2p} dx^2$ is isometric to $[0,\infty)$ equipped with the standard Euclidean one, even though it is formally not smooth at the boundary point $0$. Consequently, the manifold $(\Real_+^d,g_p)$ is isometric to $\Real_+^d$ with the Euclidean metric, and its boundary is convex. The fact that the metric formally explodes on the boundary may be corrected by applying the latter isometry; however, the boundary will remain non-smooth. This can be taken care of by either approximating $\Real_+^d$ from within by a convex set with smooth boundary, or simply by noting that $((0,\infty)^d,g_p)$ is geodesically convex, so that the Geometric Convexity Assumptions are in any case valid. When $\Omega$ is a proper subset of $(0,\infty)^d$, we may assume that it is relatively closed or open, and its geodesic (and local convexity) follow from Lemma \ref{lem:orthant-convex}. Consequently, the Geometric Convexity Assumptions are in force. 

Part (1) then follows by the Generalized Brascamp--Lieb Theorem \ref{thm:BL}. Part (2) follows by applying Part (1) with $p=1/2$, and noting that
\[
 \Ric_{g_{\frac{1}{2}},\mu} \geq \mbox{\rm{diag}}\Bigl\{ \frac{1}{4 x_i^2}  \Bigr\} .
 \]
 Part (3) follows by applying Part (1) with $p \rightarrow 1$ (or in fact $p=1$ directly), noting that
 \[
 \Ric_{g_p,\mu} \geq \mbox{\rm{diag}}\Bigl\{ \lambda \frac{p}{x_i} \Bigr\} .
 \]
 Part (4) follows by the Bakry--\'Emery criterion, since when $p \in (0,1)$, 
 \[
\frac{p (1-p)}{x^2} \geq \frac{p (1-p)}{R^{2 - 2p}} \frac{1}{x^{2p}}  \;\;\;  \forall x \in (0,R] . 
\] 
In particular, using $p=1/2$ and the trivial bound $x_i \leq R$, the last assertion of Part (4) follows. 
Part (5) follows similarly after noting that whenever $p \in [1/2,1)$,
\[
\lambda \frac{p}{x}  + \frac{p (1-p)}{x^2} \geq \brac{\frac{\lambda p}{2-2p}}^{2 - 2p} \brac{\frac{ p (1-p)}{2p-1}}^{2p-1} \frac{1}{x^{2p}} \;\;\; \forall x > 0 . 
\]
Specializing to the case $p=1/2$, the last assertion of Theorem \ref{thm:product-poly} follows. 
\end{proof}

\begin{rem}
Part (2) above was established by B. Klartag in \cite{KlartagMomentMap}; Klartag provided several alternative proofs, one of which indeed employed the classical Brascamp--Lieb inequality (\ref{eq:classical-BL}). Applied to $f(x) = \abs{x}^2$, it yields the so-called thin-shell estimate for the variance of $\abs{x}^2$ on (orthant) unconditional convex sets, first confirmed by Klartag in \cite{KlartagUnconditionalVariance} - cf. Subsection \ref{subsec:orth-to-full}. It also follows immediately from Part (2) that when $f$ is $1$-Lipschitz ($\abs{\nabla f} \leq 1$) on $\Omega$, then
\[
\Var_{\mu}(f) \le  4 \int \max_{i=1,\ldots,d} \abs{x_i}^2  d\mu .
\]
Using the known equivalence between concentration and isoperimetry on weighted manifolds with non-negative generalized Ricci curvature \cite{EMilman-RoleOfConvexity,EMilmanGeometricApproachPartI,EMilmanIsoperimetricBoundsOnManifolds}, it follows (cf. (\ref{eq:OneLip}) below) that
\[
\Var_{\mu}(f) \le  4 C \int \max_{i=1,\ldots,d} \abs{x_i}^2  d\mu \int \abs{\nabla f}^2 d\mu \;\;\; \forall f \in C^1(\Omega) ,
\]
for some universal numeric constant $C > 1$, a fact first established by Klartag in \cite{KlartagUnconditionalVariance}. 
\end{rem}

\subsection{Exponentially decaying product metrics}

We next specialize to the case $u_i(x_i) = \exp(\lambda_i x_i)$, which yields the following estimate for log-concave measures with super linear potentials. 

\begin{thm}  \label{exp-type}
Let $\mu = \exp(-V) dx$ denote a probability measure with smooth and positive density, supported in $\Omega$, an orthant unconditional convex subset of $(0,\infty)^d$. Assume that
\[
 D^2 V \geq 0 \; , \; V_{x_i} > \lambda_i > 0 \;\;\; \forall i=1,\ldots,d .
 \]
 Then the generalized Ricci tensor of the weighted manifold
\[
 \brac{\Omega  , g = \sum_{i=1}^d  \exp(-2 \lambda_i x_i)  (dx^i)^2 , \mu = \exp(-V) dx}
\]
satisfies
\[
\Ric_{g,\mu} \geq \mbox{\rm diag}\{\lambda_i (V_{x_i} - \lambda_i)\} ,
\]
and we have
\[
\Var_{\mu}(f) \le \int \sum_{i=1}^d \frac{f^2_{x_i}}{\lambda_i(V_{x_i}-\lambda_i)} d\mu \;\;\; \forall f \in C^1(\Omega). 
\]
In particular, if
\[
D^2 V \geq 0 \; , \; V_{x_i} \geq \lambda > 0 \;\;\; \forall i=1,\ldots,d ,
\]
we have
\[ \Var_{\mu}(f) \le \frac{4}{\lambda^2} \int  \abs{\nabla f}^2 d\mu \;\;\; \forall f \in C^1(\Omega). 
\] \end{thm}
\begin{proof}
The lower bound on $\Ric_{g,\mu}$ follows immediately from Proposition \ref{prop:product-Ric}. 
Note that $(\Omega,g)$ has bounded diameter due to the integrability of the metric, and so for example $(\Real_+^d , g)$ is no longer complete as in the previous subsection. However, as explained there, $(\intr(\Omega),g)$ is geodesically convex, and so the Geometric Convexity Assumptions are still in force, and the assertion follows from the Brascamp--Lieb Theorem \ref{thm:BL}. The in particular part obviously follows by setting $\lambda_i \equiv \frac{\lambda}{2}$. 
\end{proof}

\subsection{Unconditional Convex Sets With Diagonal Boundary}

Let $\Omega \subset (0,\infty)^d$ denote a relatively closed bounded set with positive volume, and let $\lambda_\Omega$ denote the normalized Lebesgue probability measure on $\Omega$. Recall that the Poincar\'e constant of $\Omega$ is defined to be the best constant $C_P(\Omega)$ in the following inequality:
\[
\Var_{\lambda_\Omega}(f) \leq C_P (\Omega) \int \abs{\nabla f}^2 d\lambda_{\Omega} \;\;\; \forall f \in C^1(\Omega) .
\]
When $\Omega$ is convex, recall that its associated gauge (or norm) functional $p_\Omega : (0,\infty)^d \rightarrow \Real$ is defined as
\[
p_\Omega(x) := \min \{ t \ge 0 ; x \in t \Omega\} .
\]
The cone measure $\sigma_{\Omega}$ is the probability measure on $\partial_r \Omega := \partial \Omega \cap (0,\infty)^d$ defined as the push forward of $\lambda_\Omega$ under the mapping
\[
x \mapsto \frac{x}{p_{\Omega}(x)}  .
\]
Clearly, every probability measure on $(0,\infty)^d$ whose density is a function of $p_\Omega$ has $\sigma_\Omega$ as its image under the same mapping. It is immediate to verify that
\begin{equation} \label{eq:cone-formula}
d\sigma_{\Omega}(x) = \frac{1}{\text{Vol}(\Omega)} \frac{\scalar{x,n}}{d} d\mathcal{H}^{d-1}|_{\partial_r \Omega} (x) ,
\end{equation}
where $n$ denotes the outer unit normal to $\partial \Omega$.

\begin{thm} \label{l1-type}
Let $\Omega \subset (0,\infty)^d$ ($d \geq 3$) denote an orthant unconditional, relatively closed, bounded convex set, so that $\partial_r \Omega$ is smooth and on it
\[
\frac{\langle n, e_i \rangle}{\langle n, x \rangle} \geq \lambda > 0 \;\;\; \forall i =1 ,\ldots, d . 
\]
Then
\[
C_P(\Omega) \leq \frac{C}{d^2}  \Bigl( {\int_{\Omega} \abs{x}^2 d \lambda_\Omega} +  \frac{1}{\lambda^2} \int_{\partial_r \Omega}\frac{\abs{x}^2}{\langle x, n \rangle ^2} d \sigma_{\Omega} \Bigr),
 \]
where $C > 0$ is a numeric constant. In particular, if
\begin{equation} \label{eq:diagonal}
0 < \lambda \le \frac{\langle n, e_i \rangle}{\langle n, x \rangle} \le \Lambda \;\;\; \forall i =1 ,\ldots, d ,
\end{equation}
then
\[ 
C_P(\Omega) \leq C \frac{1+(d+2) \brac{\frac{\Lambda}{\lambda}}^2}{d^2}  \int_{\Omega} \abs{x}^2 d \lambda_{\Omega} .
\]
\end{thm}

\begin{rem} 
The condition (\ref{eq:diagonal}) reflects the property that in some sense, $\Omega$ is close to an orthant of an $\ell_1^d$-ball, whose facet is normal to the diagonal direction. Note that applying Theorem \ref{l1-type} to $\Omega_1^d =  \{x \in \Real^d \; ; \; x_i \ge 0 \; ,\; \sum_{i=1}^d x_i \leq 1\}$, recovers (up to numeric constants) its correct Poincar\'e constant (see \cite{SodinLpIsoperimetry}). Indeed, in this case (\ref{eq:diagonal}) holds with $\lambda=\Lambda=1$, and Theorem \ref{l1-type} yields
\[
C_P(\Omega_1^d) \leq C \frac{d+3}{d^2} \int_{\Omega_1^d} \abs{x}^2 d \lambda_{\Omega_1^d} = C \frac{d+3}{d} \int_{\Omega} x^2_1 d \lambda_{\Omega_1^d},
\]
which is of the right order in $d$ (as follows e.g. from the next subsection and testing the linear function $f(x) = x_1$). 
\end{rem}

For the proof, we will require the following:

\begin{prop} \label{cone-comparis}
Under the same assumptions as in Theorem \ref{l1-type}, we have for every $1$-Lipschitz function $h$ on $\partial_r \Omega$,
\[
\Var_{\sigma_\Omega}(h) \le \frac{4}{\lambda^2 (d-1)(d-2)} \int_{\partial_r \Omega}\frac{\abs{x}^2}{\langle x, n \rangle ^2} d \sigma_\Omega .
\]
\end{prop}

\begin{proof}
Consider the probability measure $\mu = \frac{1}{Z} \exp(-p(x)) dx$ on $(0,\infty)^d$ for $p = p_\Omega$. By homogeneity $\langle x , \nabla p(x) \rangle = p(x)$, and as $n = \frac{\nabla p}{\abs{\nabla p}}$, we have for $x \in \partial_r \Omega$,
$$
\partial_{x_i} p = \langle n, e_i \rangle |\nabla p| \ge \lambda \langle n, x \rangle |\nabla p|
= \lambda \langle \nabla p,x \rangle = \lambda p(x) =\lambda.
$$
By homogeneity, we conclude that $\partial_{x_i}  p \ge \lambda$ for all $x \in (0,\infty)^d$ and $i=1,\ldots,d$. Invoking Theorem \ref{exp-type}, we deduce that $\mu$ satisfies
$$
\Var_{\mu} (f) \le \frac{4}{\lambda^2} \int |\nabla f|^2 \ d \mu ,
$$
for all $f \in C^1((0,\infty)^d)$. 

Now given a $1$-Lipschitz function $h$ on $\partial_r \Omega$, let us apply the above inequality to $f(x) := h(x/p(x))$. We obtain
$$
\Var_{\sigma_\Omega}(h)  = \Var_{\mu}(f)  \le \frac{4}{\lambda^2} \int \| dT(x) \|^2_{op} |\nabla h(T x)|^2 \ d \mu \leq \frac{4}{\lambda^2} \int \| dT(x)\|^2_{op} \ d \mu ,
$$
where $T(x) = x/p(x)$ and $\| \cdot \|_{op}$ is the corresponding operator norm. To estimate the latter integral, let us represent $\mu$ as the product of probability measures
 $$
 d\mu(x) = \frac{1}{(d-1)!} t^{d-1} e^{-t} dt \otimes d\sigma_\Omega(y) ,
 $$
 where $(t,y) = (p(x),x / p(x)) \in (0,\infty) \times \partial_r \Omega $ are adapted ``polar coordinates". By \cite[Lemma 15]{KolesnikovEMilman-HardyKLS}, we know that
\[
\|dT(x)\|_{op} =\frac{|x|}{p(x) \langle x, n(T(x)) \rangle} = \frac{\abs{y}}{t \scalar{y,n(y)}} ,
\]
and therefore
\[
 \int \| dT(x) \|^2_{op} \ d \mu
 = \frac{1}{(d-1)!} \int_0^{\infty} t^{d-3} e^{-t} dt \; \cdot \int_{\partial_r \Omega} \frac{\abs{y}^2}{\scalar{y,n}^2} d\sigma_{\Omega} ,
 \]
 thereby concluding the proof. 
\end{proof}

\begin{proof}[Proof of Theorem \ref{l1-type}]
By a Hardy-type estimate from \cite[Theorem 1]{KolesnikovEMilman-HardyKLS}, we know that for all $f \in C^1(\Omega)$,
\[
\int_{\Omega} f^2 d\lambda_\Omega \le \frac{4}{d^2} \int_{\Omega} \scalar{ x, \nabla f }^2  d\lambda_\Omega
+ \frac{2}{d} \frac{1}{\text{Vol}(\Omega) }\int_{\partial \Omega} \scalar{ x, n}  f^2 d \mathcal{H}^{d-1}(x) .
\]
Since $\scalar{x,n} = 0$ on $\partial \Omega \setminus \partial_r \Omega$, and using (\ref{eq:cone-formula}), we see that for any $1$-Lipschitz function $f$ on $\Omega$ (and in particular, on $\partial_r \Omega$),
\[
\Var_{\lambda_\Omega}(f) \leq \frac{4}{d^2} \int \abs{x}^2 d\lambda_\Omega + 2 \Var_{\sigma_\Omega}(f) . 
\]
The first assertion of Theorem \ref{l1-type} then follows by evaluating the last integral above using Proposition \ref{cone-comparis}, and using the fact \cite{EMilman-RoleOfConvexity} that for convex domains $\Omega$,
\begin{equation} \label{eq:OneLip}
C_P(\Omega) \leq C' \sup \set{ \Var_{\lambda_\Omega}(f) \; ; \; \text{$f$ is $1$-Lipschitz} } 
\end{equation}
for some numeric constant $C' > 1$ (cf. \cite[Corollary 8]{KolesnikovEMilman-HardyKLS}). 

To verify the second assertion, we simply use the trivial estimate
\[
\frac{1}{\langle x , n \rangle^2} = \sum_{i=1}^d \frac{\langle e_i , n \rangle^2}{\langle x , n \rangle^2} \le d \cdot \Lambda^2 .
\]
It remains to note that (e.g. integrating in polar-coordinates)
\[
\int_{\partial_r \Omega} \abs{x}^2 d \sigma_{\Omega} = \int_{\Omega} \frac{\abs{x}^2}{p_{\Omega}^2(x)} d\lambda_{\Omega} = \brac{1 + \frac{2}{d}} \int_{\Omega} \abs{x}^2 d\lambda_{\Omega} .
\]
The second assertion of Theorem \ref{l1-type} now readily follows. 
\end{proof}

\subsection{From Orthant to Unconditional log-concave measures} \label{subsec:orth-to-full}

In \cite{KlartagMomentMap}, B. Klartag showed how to transfer weighted Poincar\'e inequalities from orthant unconditional convex sets and log-concave measures, to unconditional ones defined on the entire space. We summarize his result as follows:

\begin{thm}[Klartag]
Let $\mu$ denote a log-concave unconditional probability measure on $\Real^d$, and denote by $\mu_+$ its conditioning onto $\Real^d_+$. Assume that
\begin{equation} \label{eq:orthant-var}
\Var_{\mu_+}(f) \leq \int Q_x(\nabla f) d\mu_+(x) \;\;\; \forall f \in C^1(\Real^d_+) .
\end{equation}
for some measurable quadratic form $x \mapsto Q_x$. Then
\[
\Var_{\mu}(f) \leq \int Q_x(\nabla f) d\mu(x) + \max_{i=1,\ldots,d} \int x_i^2 d\mu(x)  \int \abs{\nabla f}^2 d\mu \;\;\; \forall f \in C^1(\Real^d) .
\]
\end{thm}

For completeness, let us sketch a somewhat simplified proof of Klartag's result, which avoids the use of the dual $H^{-1}$-norm or Fourier arguments as in \cite{KlartagMomentMap}. 

\begin{proof}[Proof Sketch]
Write $\mu = \int \mu_\alpha d\nu(\alpha)$, where $\nu$ denotes the uniform measure on the discrete cube $\set{-1,1}^d$ (which we naturally identify with the $2^d$ orthants), and $\mu_\alpha$ denotes the measure $\mu$ conditioned on the $\alpha$-th orthant. Clearly,
\[
\Var_\mu(f) = \E_{\nu} \Var_{\mu_\alpha}(f)  + \Var_{\nu} \E_{\mu_\alpha}(f) . 
\]
The first term is governed by our assumption (\ref{eq:orthant-var}), whereas for the second term we employ the Poincar\'e inequality on the discrete cube:
\[
\Var_{\nu} (h) \leq \frac{1}{4} \int \sum_{i=1}^d \abs{D_i h}^2 d\nu ,
\]
where $D_i h(\alpha) = h(\alpha) - h(T_i \alpha)$ denotes the discrete partial derivative in the $i$-th direction (and $T_i$ denotes the corresponding reflection). It remains to establish that
\[
\int \abs{\E_{\mu_{\alpha}}(f) - \E_{\mu_{T_i \alpha}}(f)}^2 d\nu(\alpha) \leq 4 \int x_i^2 d\mu(x) \int \abs{\partial_{x_i} f}^2 d\mu .
\]
Since the density of $\mu$ remains unconditional and log-concave when restricted to any line in the coordinate directions, by an application of Fubini and Cauchy--Schwarz, the proof is reduced to establishing the following one-dimensional inequality for any log-concave even measure $\eta$ on $\Real$:
\[
\abs{ \int f(x_i) d\eta(x_i) - \int f(-x_i) d\eta(x_i) } \leq 2 \sqrt{ \int x_i^2 d\eta(x_i) } \sqrt{ \int \abs{f'}^2 d\eta} \;\;\; \forall f \in C^1(\Real) .
\]
Using only the unimodality of $\eta$, integration by parts and Cauchy--Schwartz, verification of the latter is elementary (see \cite[Lemma 5.1]{KlartagMomentMap}). 
\end{proof}

Unfortunately, 
the above argument does not work for transferring a weighted log-Sobolev inequality from an orthant to the entire space (although various inequalities, not as elegant as the one for weighted Poincar\'e, may be derived).

\section{Conformal Metrics} \label{sec:conformal}

In this section, we no longer restrict our discussion to Hessian metrics, and consider metrics which are conformal to the standard Euclidean one.  Conformal changes of metric are very important transformations in geometry and mathematical physics, and we aim to demonstrate their usefulness in our context as well.

\subsection{Brascamp--Lieb Inequality for finite Generalized Dimension}

For the results of this section, we will require the following additional notions. Recall that we are working on a weighted manifold $(M^d,g,\mu = \exp(-V) \vol_g)$.  

The generalized weighted $N$-dimensional Ricci tensor, $N \in (-\infty,\infty]$, originally introduced by Bakry \cite{BakryStFlour} to obtain a tensorial equivalent formulation of the Bakry--\'Emery Curvature-Dimension condition \cite{BakryEmery}, is defined as follows:
\[
\Ric_{g,\mu,N} := \Ric_{g,\mu} - \frac{1}{N - d} \nabla_g V \otimes \nabla_g V . 
\]
Given a smooth hypersurface $S \subset M$ with unit normal $n$, its generalized weighted mean-curvature is defined as
\[
H_{g,\mu} := \text{tr}_g(\II_{S}) - \scalar{\nabla V,n} ,
\]
where $\II_S$ denotes the second fundamental form of $S$ with respect to $n$. Finally, the weighted volume measure on $S$ is given by
\[
\mu_{S,g} := \exp(-V) \vol_{S,g} ,
\]
where $\vol_{S,g}$ denotes the Riemannian volume measure on $S$ with respect to the metric induced from $g$. 

\medskip

The results of this section are based on the following extension of the Brascamp-Lieb inequality in the Riemannian setting with boundary, obtained by the authors in \cite{KolesnikovEMilman-Reilly}.

\begin{thm}[Generalized Dimensional Brascamp--Lieb With Boundary] \label{BL-Dir} \hfill \\
Let $(M^d,g,\mu = \exp(-V) \vol_g)$ denote a weighted manifold which is assumed oriented and compact. Let $\frac{1}{N} \in (-\infty, \frac{1}{d}]$, and assume that $\Ric_{g,\mu,N} > 0$ on $M$. Then for any $f \in C^1(M)$:
\begin{enumerate}
\item When $M$ is strictly (generalized) mean-convex: \\
If $H_{g,\mu} > 0$ on $\partial M$, then
\[ 
\frac{N}{N-1} \Var_\mu(f) \leq \int_M \langle{ \Ric_{g,\mu,N}^{-1} \; \nabla f , \nabla f} \rangle
d\mu + \min_{C \in \Real} \int_{\partial M} \frac{1}{H_{g,\mu}} (f - C)^2 d\mu_{\partial M,g} .
\]
\item
When $M$ is locally convex: \\
If $\II_{\partial M} \geq 0$ on $\partial M$, then
\[
\frac{N}{N-1} \Var_\mu(f) \leq \int_M \langle{ \Ric_{g,\mu,N}^{-1} \; \nabla f , \nabla f} \rangle
d\mu.
\]
\item 
When $M$ is (generalized) mean-convex and under Dirichlet boundary conditions:\\
If $H_{g,\mu} \geq 0$ on $\partial M$ and $f|_{\partial M} \equiv 0$, then
\[
\frac{N}{N-1} \int_M f^2 d\mu \leq \int_M \langle{ \Ric_{g,\mu,N}^{-1} \; \nabla f , \nabla f} \rangle d\mu . 
\]
\end{enumerate}
\end{thm}

\subsection{Conformal Transformation}

We will need below a list of formulae for the change of various geometric quantities under conformal transformation. Some of them may be found in \cite{Besse-EinsteinManifoldsBook}, the others may be easily derived by direct computation. Assume we are given a conformal change
\[
g = e^{2 \varphi} \cdot g_0
\]
of a Riemannian metric $g_0$ on the manifold $M$. The symbols $\nabla, \Delta, \Gamma_0, D^2, \langle \cdot, \cdot \rangle$ below are understood with respect to the metric $g_0$. 

\begin{enumerate} 
\item 
Riemannian volume measure:
\[
\vol_{g} = \exp( d \varphi) \vol_{g_0} . 
\]
\item
Christoffel symbols: 
$$
\Gamma^{m}_{ij} = 
(\Gamma_0)^{m}_{ij} +  \delta^m_{i} \cdot \varphi_{x_j} + \delta^m_{j} \cdot \varphi_{x_i}
- (g_0)_{ij} \nabla^m \varphi.
$$
\item
Unweighted (geometric) Ricci tensor: $$
\Ric_g = 
\Ric_{g_0} -(d-2) \Bigl( D^2 \varphi - \nabla \varphi \otimes \nabla \varphi \Bigr) - \bigl( \Delta \varphi + (d-2) |\nabla \varphi|^2\bigr) g_0.
$$
\item
Hessian of a function $f$: $$
\mbox{Hess}_g (f) =  D^2 f - \nabla \varphi \otimes \nabla f - \nabla f \otimes \nabla \varphi + \langle \nabla \varphi, \nabla f \rangle \cdot g_0.
$$
\item
The generalized Ricci tensor of the weighted manifold
\begin{equation} \label{eq:conformal-weighted-manifold}
 (M, g, \mu = \exp(-V) \vol_{g_0} = \exp(-(V + d \varphi)) \vol_g)
\end{equation}
is given by
\begin{align*}
\Ric_{g,\mu} &= \Ric_{g} + \Hess_{g}(V + d \varphi) \\
& = \Ric_{g_0} + \Hess_g V +  2 D^2 \varphi - (d+2) \nabla \varphi \otimes \nabla \varphi + (2 \abs{\nabla \varphi}^2 - \Delta \varphi) \cdot g_0.
\end{align*}
\item 
Generalized $N$-dimensional Ricci tensor of the latter weighted manifold:
\begin{align} 
\nonumber  \Ric_{g,\mu,N}  & =  \Ric_{g,\mu}  +  \frac{1}{d-N} (\nabla  V + d \nabla \varphi) \otimes (\nabla  V + d \nabla \varphi)  \\
\nonumber &  = \Ric_{g_0} + D^2 V + \langle \nabla V, \nabla \varphi \rangle \cdot g_0 + \frac{1}{d-N} \nabla V \otimes \nabla V + \frac{N}{d-N} \bigl( \nabla V \otimes \nabla \varphi +  \nabla \varphi \otimes \nabla V \bigr) \\
\label{conform-gen-ric}  &  + \Bigl( \frac{dN}{d-N} -2 \Bigr) \nabla \varphi \otimes \nabla \varphi  + 2 D^2 \varphi + (2 |\nabla \varphi|^2 - \Delta \varphi) \cdot g_0.
\end{align}
\item
Given a smooth hypersurface $S\subset M$, we clearly have
\begin{align} 
\nonumber \vol_{g}|_{S} & = \exp((d-1)\varphi) \vol_{g_0}|_{S} ,\\
\label{eq:conformal-boundary-mes} \mu_{S,g} = \exp(-(V + d \varphi)) \vol_{g}|_{S} & = \exp(-(V + \varphi)) \vol_{g_0}|_{S} = \exp(-\varphi) \mu_{S,g_0} .
\end{align}
\item
Let $n_{0}$ denote a unit normal to $S$ in the $g_0$ metric. Note that $n_{0}$ remains perpendicular to $S$ in the conformal $g$ metric, but the unit normal vector is now $n_{g} = \exp(-\varphi) n_{0}$. We then obtain the following formulae for the second fundamental form $\II_g$ (with respect to $n_g$ and $g$) and the generalized mean curvature $H_{g,\mu}$ of $S$ in the weighted manifold (\ref{eq:conformal-weighted-manifold}): \begin{equation} \label{II}
\II_g   = e^{\varphi} \Bigl( \II_{0}   +   \langle \nabla \varphi, n_0\rangle  g_{0}|_{S}\Bigr),
\end{equation}
\begin{equation} \label{H}
H_{g,\mu} = e^{-\varphi} \Bigl( H_{0}  -  \langle  \nabla \varphi + \nabla V, n_0 \rangle \Bigr). 
\end{equation}
Here $\II_0$ and $H_0 = \text{tr}_{g_0}(\II_0)$ denote the second fundamental form and the unweighted (geometric) mean curvature with respect to $n_0$ and the $g_0$ metric. 
\end{enumerate}

\subsection{Spectral-Gap via Mean-Curvature and Boundary Angle}

The main result of this section is the following: 
\begin{thm} \label{31.10.2013}
Let  $\Omega \subset \Real^d$ denote a compact convex set with smooth boundary containing the origin in its interior, and let $\lambda_{\Omega}$ denote 
the normalized Lebesgue probability measure on $\Omega$. Let $N \leq 0$, and assume that
\begin{equation} \label{eq:small-cond}
\brac{\frac{1}{2} - \frac{1}{N}} d \geq 3 
\end{equation}
(in particular, this holds if $d \geq 6$). Then for any $f \in C^1(\Omega)$, we have
\[
\frac{1}{1-N} \Var_{\lambda_{\Omega}}(f) \le   \frac{4}{d(d-N)} \int_{\Omega} \abs{x}^2 |\nabla f|^2 d \lambda_{\Omega} + \frac{1}{\mbox{\rm{Vol}}(\Omega)} \int_{\partial \Omega} 
\frac{1}{ \frac{d-N}{2} \frac{ \langle  x, n\rangle}{ \abs{x}^2}  -N H_0(x)   } (f -C)^2 d \mathcal{H}^{d-1}.
\]
Here $C \in \Real $ is an arbitrary constant,  $n$ is the outer unit normal to $\partial \Omega$, and $H_0(x)$ is the Euclidean mean-curvature of $\partial \Omega$ at $x$. 
\end{thm}

Of course one should optimize on the location of the origin to obtain the best possible (translation invariant) estimate above. 

\begin{proof}
Consider the weighted manifold $(\Omega, g = e^{2\varphi} \cdot g_0, \lambda_{\Omega})$, where $g_0$ is the standard Euclidean metric, with
\begin{equation} \label{eq:conformal-poly}
\varphi = - \frac{\theta}{2} \log (\abs{x}^2 + \eps),
\end{equation}
where $\eps >0$ will be sent to $0$ and the parameter $\theta \ge 0$ will be chosen later.
According to (\ref{conform-gen-ric}),
\[
\Ric_{g,\lambda_{\Omega},N}  = 
  \Bigl( \frac{dN}{d-N} -2 \Bigr) \nabla \varphi \otimes \nabla \varphi
   + 2 D^2 \varphi + (2 |\nabla \varphi|^2 - \Delta \varphi) Id .
\]
Plugging in the following expressions for $\nabla \varphi,  D^2 \varphi, \Delta \varphi$:
\[
\nabla \varphi = - \theta \frac{x}{\abs{x}^2 + \eps}, \ D^2 \varphi = - \theta \Bigl( \frac{Id}{\abs{x}^2+\eps} - 2 \frac{x \otimes x}{(\abs{x}^2+\eps)^2}\Bigr) ,
\]
\[
\Delta \varphi = - \theta \Bigl(\frac{d}{\abs{x}^2+\eps} - \frac{2 \abs{x}^2}{(\abs{x}^2+\eps)^2}\Bigr) ,
\]
we obtain (for all $N \in (-\infty,\infty]$)
\begin{align*}
\Ric_{g,\lambda_{\Omega},N} &  =
\Bigl[ \theta^2 \Bigl( \frac{dN}{d-N} -2 \Bigr) + 4 \theta \Bigr]\frac{x \otimes x}{(\abs{x}^2 + \eps)^2}
+ \frac{1}{\abs{x}^2 + \eps} \Bigl[  \theta (d-2) +  \frac{2\theta (\theta -1 ) \abs{x}^2}{\abs{x}^2 + \eps} \Bigr] Id \\
& = \frac{1}{\abs{x}^2} \brac{
 \theta (d + 2 \theta - 4)\Bigl( Id - \frac{x}{\abs{x}} \otimes \frac{x}{\abs{x}} \Bigr) + \brac{d \theta + \frac{dN}{d-N} \theta^2} \frac{x}{\abs{x}} \otimes \frac{x}{\abs{x}} } (1+o(\eps)).
\end{align*}
Clearly,  the tensor is non-negative as $\eps \rightarrow 0$  if and only if
\[
d +  2 \theta - 4 \ge 0 \; \text{ and } \;
1   +    \theta \frac{N}{d-N}  \ge 0.
\]
Assuming that $N < 0$, and setting $\theta$ so as to maximize the generalized Ricci curvature in the radial direction:
\[
\theta = - \frac{d-N}{2N} > 0 ,
\]
one easily verifies that when (\ref{eq:small-cond}) holds, then
\[
\Ric_{g,\lambda_{\Omega},N} \ge - \frac{d(d-N)}{4N \abs{x}^2} Id \; (1 + o(\eps)) .
\]
Lastly, we obtain from (\ref{H}) and (\ref{eq:conformal-boundary-mes}) the following expressions for the generalized mean curvature and boundary volume measure:
\[
H_{g,\lambda_{\Omega}}(x) =
(\abs{x}^2 + \eps)^{\frac{\theta}{2}} \Bigl( H_0(x)   - \frac{d-N}{2 N}   \frac{\langle  x, n\rangle}{\abs{x}^2 + \eps}   \Bigr) ~,~ (\lambda_{\Omega})_{\partial \Omega,g} = (\abs{x}^2 + \eps)^{\frac{\theta}{2}} \mathcal{H}^{d-1}|_{\partial \Omega} .
\]
Since $H_0 \geq 0$, $N < 0$ and $\scalar{x,n} > 0$ (as the origin lies in the interior of the convex $\Omega$), we confirm that $H_{g,\lambda_{\Omega}} > 0$ on $\partial \Omega$. 

Applying the first assertion of Theorem \ref{BL-Dir} and taking limit as $\eps \to 0$, the asserted estimate follows. The case $N=0$ is obtained by passing to the limit. 
\end{proof}

\begin{rem}
For functions $f$ which vanish on $\partial \Omega$, one may apply the third assertion of Theorem \ref{BL-Dir} instead of the first one, and recover (using the optimal $N=0$) the following classical result of Hardy \cite{Ghoussoub-Book}:
\[
\int_{\Omega} f^2 dx \le   \frac{4}{d^2} \int_{\Omega} \abs{x}^2 |\nabla f|^2 dx .
\] 
Consequently, Theorem \ref{31.10.2013} may be thought of as an extension of Hardy's inequality for general functions with boundary term.
\end{rem}

\begin{rem}
Setting $N=0$, Theorem \ref{31.10.2013} yields
\begin{equation} \label{eq:Hardy-new}
\Var_{\lambda_{\Omega}}(f) \leq \frac{4}{d^2} \int_{\Omega} \abs{x}^2 |\nabla f|^2 d \lambda_{\Omega} + \frac{2}{\mbox{\rm{Vol}}(\Omega)} \min_{C \in \Real} \int_{\partial \Omega}  \frac{\abs{x}^2}{ d \scalar{x, n}} (f -C)^2 d \mathcal{H}^{d-1} .
\end{equation}
We remark that the following strictly stronger result was obtained by other methods in \cite{KolesnikovEMilman-HardyKLS}:
\begin{equation} \label{eq:Hardy-old}
\Var_{\lambda_{\Omega}}(f) \le \frac{4}{d^2} \int_{\Omega} \scalar{ x, \nabla f}^2  d\lambda_\Omega 
+ 2 \Var_{\sigma_{\Omega}}(f|_{\partial \Omega}) ,
\end{equation}
where recall $\sigma_{\Omega}$ denotes the cone probability measure, which by (\ref{eq:cone-formula}) satisfies
\[
d\sigma_{\Omega} = \frac{1}{\rm{Vol}(\Omega)} \frac{\scalar{ x, n}}{d} \cdot d\mathcal{H}^{d-1}|_{\partial \Omega}.
\]
Note that
\[
\scalar{ x, \nabla f}^2 \leq \abs{x}^2 \abs{\nabla f}^2 ~,~ \scalar{ x, n} \leq \frac{\abs{x}^2}{\scalar{x,n}} ,
\]
elucidating the relation between (\ref{eq:Hardy-new}) and (\ref{eq:Hardy-old}). The main new case of interest in Theorem \ref{31.10.2013} is therefore when $N < 0$. 
\end{rem}

\begin{rem}
As in \cite{KolesnikovEMilman-HardyKLS}, we may use Theorem \ref{31.10.2013} to reduce the problem 
of bounding above the Poincar\'e constant $C_P(\Omega)$, to controlling the variance of $1$-Lipschitz functions on its boundary. Indeed, setting $N=-1$, we have for such functions (since $\abs{\nabla f} \leq 1$)
\[
\Var_{\lambda_\Omega}(f) \le \frac{8}{d(d+1)} \int_{\Omega} \abs{x}^2  d \lambda_{\Omega}
+ \frac{1}{\mbox{\rm{Vol}}(\Omega)} \int_{\partial \Omega} 
\frac{2}{H_0   + \frac{d+1}{2}  \frac{\langle  x, n\rangle}{\abs{x}^2}} (f -C)^2 d \mathcal{H}^{d-1} .
\]
Assume that $\Omega$ is isotropic, i.e. that the covariance of a random vector uniformly distributed in $\Omega$ is the identity matrix. Then $\frac{1}{d^2} \int_{\Omega} \abs{x}^2  d \lambda_{\Omega} = \frac{1}{d}$, which is negligible with respect to the variance of the worst $1$-Lipschitz function (by e.g. testing linear ones). 
Using (\ref{eq:OneLip}), it follows as in \cite{KolesnikovEMilman-HardyKLS} that
\[
C_P(\Omega)\leq C' \sup \set{ \frac{1}{\mbox{\rm{Vol}}(\Omega)} \min_{C\in \mathbb{R}}\int_{\partial \Omega} 
\frac{1}{H_0   + \frac{d+1}{2}  \frac{\langle  x, n\rangle}{\abs{x}^2} } (f -C)^2 d \mathcal{H}^{d-1} ; \text{$f$ is $1$-Lipschitz}} ,
\]
where $C'>0$ is a numeric constant. \end{rem}

\begin{rem}
The method described in this subsection is very general, and is not restricted to radial conformal metric changes. For instance, let us consider  $(B_p^d,g,\lambda_{\Omega_p})$, where $B^d_p = \{ \sum_{i=1}^{d} |x_i|^p \le 1\}$ is the unit-ball of $\ell_p^d$,
and $g = \exp(c_p \sum_{i=1}^{d} |x_i|^p)\cdot g_0$ is conformal to the Euclidean metric $g_0$ (with $c_p > 0$ a sufficiently small constant). We can show that this weighted manifold's generalized Ricci tensor is bounded from below by $c'_p d^{2/p} g_0$ if $1 < p  \le 2$.
A simpleminded application of the Brascamp--Lieb inequality seems to immediately recover (up to a dimension independent constant) the known estimate on the Poincar\'e constant of $B_p^d$ \cite{SodinLpIsoperimetry}, but in fact this observation is not applicable directly because $\partial B_p^d$ is not $g$-locally-convex. This obstacle can be bypassed as in this subsection, by employing the first assertion of Theorem \ref{BL-Dir} in the case of a (generalized) mean-convex boundary. Another approach is given in \cite{KolesnikovEMilman-HardyKLS}. 
\end{rem}

\subsection{Spectral Gap for Sets with Strongly Convex Boundary}

Repeating the proof of Theorem \ref{31.10.2013} and applying (when applicable) positivity of the second fundamental form, we obtain the following result:

\begin{thm}
Let  $\Omega \subset \Real^d$, $d \geq 8$, denote a compact set with smooth boundary so that $0 \notin \partial \Omega$, and let $\lambda_{\Omega}$ denote 
the normalized Lebesgue probability measure on $\Omega$. Assume that the (Euclidean $g_0$) second fundamental form on $\partial \Omega$ satisfies
\[
\II_{\partial \Omega, g_0} \ge  \theta \frac{\scalar{x, n}}{\abs{x}^2} \cdot \mbox{\rm{Id}}|_{\partial \Omega},
\]
for some $0 < \theta \leq 1/2$. 
Then for all $f \in C^1(\Omega)$,
\[
\Var_{\lambda_\Omega}(f) \le \frac{2}{d \theta} \int_{\Omega} \abs{x}^2 |\nabla f|^2 d \lambda_\Omega.
\]
Moreover,
\[
\Ent_{\lambda_\Omega}(f^2) \le \frac{4 (\max_{x \in \Omega} |x|)^{2(1-\theta)}}{d \theta}  \int_{\Omega} |x|^{2\theta} |\nabla f|^2 d \lambda_\Omega.
\]
\end{thm}
\begin{proof}
Repeating the proof of Theorem \ref{31.10.2013}, we employ the conformal transformation
\[
g = (\abs{x}^2 + \eps)^{-\theta } g_0 ,
\]
given by (\ref{eq:conformal-poly}) using $\theta$ from the hypothesis. By (\ref{II}), we see that
\[
\II_{\partial \Omega, g} = (\abs{x}^2 + \eps)^{-\frac{\theta}{2}} \brac{ \II_{\partial \Omega,g_0} - \theta \frac{\scalar{x,n}}{\abs{x}^2 + \eps} \mbox{\rm{Id}}|_{\partial \Omega} } ,
\]
so that for $\eps > 0$ small enough (slightly modifying $\theta$ if necessary), $\partial M$ is $g$-locally-convex. 
Since
\[
\Ric_{g,\lambda_\Omega,\infty} = 
\frac{1}{\abs{x}^2} \brac{
 \theta (1-\theta) d \cdot \frac{x}{\abs{x}} \otimes \frac{x}{\abs{x}} + \theta ( 2\theta +d - 4)\Bigl( {\mbox{\rm{Id}}}  - \frac{x}{\abs{x}} \otimes \frac{x}{\abs{x}} \Bigr)  } (1+o(\eps)),
\]
we see that whenever $0 < \theta \leq 1/2$ and $d\geq 8$, 
\[
\Ric_{g,\lambda_\Omega } = \Ric_{g,\lambda_\Omega,\infty}  \geq \frac{\theta d}{2 \abs{x}^2} \cdot {\mbox{\rm{Id}}}  \;  (1+o(\eps)) . 
\]
The first assertion then follows by applying the second part of Theorem \ref{BL-Dir} (with $N=\infty$), and passing to the limit as $\eps \rightarrow 0$.

The second assertion follows by applying the Bakry--\'Emery criterion (Theorem \ref{thm:BE}) for locally convex manifolds, after noting that
\[
\Ric_{g,\lambda_\Omega} \geq \frac{\theta d}{2 \abs{x}^2} \cdot  {\mbox{\rm{Id}}} \; (1+o(\eps)) \ge \frac{\theta d}{2 (\max_{x \in \Omega} |x|)^{2(1-\theta)}} \cdot  g \; (1+o(\eps)) . 
\]
\end{proof}

\begin{rem}
Similar results to those in this section may be obtained by using other conformal metrics, such as $g = \exp(- 2 \theta \abs{x}) g_0$ and $g = \exp(- \theta \abs{x}^2) g_0$. 
\end{rem}

\section*{Appendix}  
\renewcommand{\thesection}{A}
\setcounter{thm}{0}
\setcounter{equation}{0}  \setcounter{subsection}{0}

In the Appendix, we provide rigorous proofs of several key inequalities used in this work. 

\medskip

For the reader's convenience, we first state the following particular case of Klartag's recently obtained Riemannian version of the localization method \cite{KlartagLocalizationOnManifolds}.

\begin{thm}[Klartag's Localization] \label{thm:localization}
Let $(M^d,g,\mu = \exp(-V) \vol_g)$ denote a geodesically convex weighted manifold (without boundary). 
Let $f : M \rightarrow \Real$ denote a $\mu$-integrable function so that $\int_M f d\mu = 0$ and $\int \abs{f(x)} d(x,x_0) d\mu < \infty$ for some $x_0 \in M$. Then there exists a Lebesgue measurable set $S \subset M$, a measure space $(\Lambda,\mathcal{F},\nu)$ and a collection of Lebesgue measures $\set{\mu_\alpha}_{\alpha \in \Lambda}$ on $(M,g)$ so that:
\begin{enumerate}
\item $f \equiv 0$ $\mu$-a.e. on $M \setminus S$. 
\item $\set{\mu_\alpha}$ is a disintegration of $\mu$ on S. \\
Namely, for any Lebesgue measurable set $A \subset M$:
\begin{enumerate}
\item
 The function $\alpha \mapsto \mu_\alpha(A)$ is well-defined  $\nu$-a.e. and is $\nu$-measurable. 
\item 
$\mu(A \cap S) = \int \mu_\alpha(A) d\nu(\alpha)$. 
\end{enumerate}
\item For $\nu$-a.e. $\alpha \in \Lambda$, $\mu_\alpha$ is $f$-balanced: $\int f d\mu_{\alpha} = 0$. \\
In particular, $f$ is $\mu_\alpha$ integrable for all such $\alpha \in \Lambda$. 

\item For $\nu$-a.e. $\alpha \in \Lambda$, $\mu_\alpha$ is a needle: \\
There exists an open (possibly infinite) interval $L_\alpha \subset \Real$ endowed with a measure $\eta_\alpha = \exp(-V_\alpha(t)) dt$, with $V_\alpha$ smooth on $L_\alpha$, and a distance-minimizing smooth geodesic curve $\gamma_\alpha: L_\alpha \rightarrow M$ parametrized by arc-length, so that $\gamma_\alpha$ pushes-forward $\eta_\alpha$ onto $\mu_\alpha$. 
Furthermore, on $L_\alpha$, $V_\alpha$ satisfies
\[
V_\alpha''(t) \geq \Ric_{g,\mu}(\gamma_\alpha'(t),\gamma_\alpha'(t))  . 
\]
\end{enumerate}
\end{thm}

\begin{rem}
In fact, Klartag shows that $\set{\gamma_\alpha(L_\alpha)}_{\alpha \in \Lambda}$ is a partition of $S$, up to a $\mu$ null-set. 
Note that the Lebesgue $\sigma$-field is well-defined on a smooth differentiable manifold (see \cite{KlartagLocalizationOnManifolds} or just define it as the completion of the Borel $\sigma$-field). 
\end{rem}

We now proceed to give a proof of Theorem \ref{thm:BL}, which we recall here for convenience:
\begin{thm}[Generalized Brascamp--Lieb Inequality]  \label{thm:BL-appendix}
Let $(M,g,\mu)$ denote a weighted Riemannian manifold satisfying the Geometric Convexity Assumptions. Assume that $\Ric_{g,\mu} > 0$ on $M$. 
Then for all $f \in C^1(M)$,
\[
\Var_\mu(f) \leq \int_M \scalar{\Ric_{g,\mu}^{-1} \nabla f,\nabla f} d\mu .
\]
\end{thm}
\begin{proof}
We may clearly assume that in addition $\int f d\mu = 0$ and that $\int f^2 d\mu < \infty$. 

Let us first treat Case (2) of the Geometric Convexity Assumptions, namely that $\intr(M) = M \setminus \partial M$ is geodesically convex. Since by definition $\mu(\partial M) = 0$,
we may clearly restrict to $\intr(M)$ and assume that $M$ is without boundary. As explained in the proof of \cite[Lemma 6.13]{EMilman-RoleOfConvexity}, the condition that $\Ric_{g,\mu} \geq 0$ ensures by a Riemannian Borell-type lemma based on the Riemannian Brunn--Minkowski inequality of \cite{CMSInventiones}, that $\mu$ has exponential tail decay, and so in particular $\int d(x,x_0)^2 d\mu(x) < \infty$ for all $x_0 \in M$; indeed, neither orientability nor completeness are required for invoking the latter Brunn--Minkowski inequality --- the only crucial property is geodesic convexity. 

Since $f \in L^2(\mu)$, it follows by Cauchy-Schwarz that $\int \abs{f} d(x,x_0) d\mu < \infty$, and so we may invoke Theorem \ref{thm:localization}, proceeding with the notation used there. Note that for $\nu$-a.e. $\alpha$, since $\int f d\mu_\alpha = 0$, we have by the one-dimensional classical Brascamp--Lieb inequality (\ref{eq:classical-BL}),
\begin{align*}
\int f^2 d\mu_\alpha & = \int_{L_\alpha} f^2(\gamma_\alpha(t)) d\eta_\alpha(t) \\
& \leq \int_{L_\alpha} (V_\alpha''(t))^{-1} (\frac{d}{dt} f(\gamma_\alpha(t)))^2 d\eta_\alpha(t) \\
& \leq \int_{L_\alpha} \Ric_{g,\mu}(\gamma_\alpha'(t),\gamma_\alpha'(t))^{-1} \scalar{\nabla f (\gamma_\alpha(t)),\gamma_\alpha'(t)}^2 d\eta_\alpha(t) .
\end{align*}
Applying Cauchy--Schwarz for the positive-definite form $\Ric_{g,\mu}$,
\[
\scalar{\nabla f (\gamma_\alpha(t)),\gamma_\alpha'(t)}^2 \leq \Ric_{g,\mu}(\gamma_\alpha'(t), \gamma_\alpha'(t)) \cdot \Ric_{g,\mu}^{-1}(\nabla f (\gamma_\alpha(t)),\nabla f (\gamma_\alpha(t))) ,
\]
we deduce that
\[
\int f^2 d\mu_\alpha \leq \int \Ric_{g,\mu}^{-1}(\nabla f,\nabla f) d\mu_\alpha .
\]
Integrating over $\alpha$ with respect to $\nu$, we obtain
\[
\int_{S} f^2 d\mu \leq \int_S \Ric_{g,\mu}^{-1}(\nabla f,\nabla f) d\mu  .
\]
But since $f \equiv 0$ on $M \setminus S$, we obtain
\[
\int_{M} f^2 d\mu \leq \int_M \Ric_{g,\mu}^{-1}(\nabla f,\nabla f) d\mu ,
\]
as asserted.

We now turn to Case (1) of the Geometric Convexity Assumptions, namely that $M$ is oriented, complete, with locally-convex (possibly empty) smooth boundary.
Integration by parts immediately verifies that $-\Delta_{g,\mu}$ is a positive semi-definite symmetric operator on $L^2(\mu)$ with dense domain $C_{c,Neu}^\infty(M)$, the subspace of compactly supported smooth functions satisfying vanishing Neumann boundary conditions. It is known \cite[Chapter 8]{Taylor-PDE-II-Book} that in fact $-\Delta_{g,\mu}$ is essentially self-adjoint on the latter domain, so that its closure in the graph-norm is its unique self-adjoint extension. Since $M$ is connected, it is well-known and easy to see that $E_0$, the orthonormal spectral projection onto the eigenspace of $-\Delta_{g,\mu}$ corresponding to the $0$ eigenvalue, consists of projection onto the one-dimensional subspace of constant functions. The essential self-adjointness then implies that
\[
\overline{-\Delta_{g,\mu}(C_{c,Neu}^\infty(M))} = \Img(E_0)^\perp 
\]
as linear subspaces of $L^2(\mu)$. 
Consequently, for any $f \in L^2(\mu)$ so that $\int f d\mu = 0$, there exists a sequence $u_k \in C_{c,Neu}^\infty(M)$ so that $-\Delta_{g,\mu} u_k \rightarrow f$ in $L^2(\mu)$. This is precisely what is required to repeat the proof of the Brascamp-Lieb inequality from the \emph{compact} Riemannian setting of \cite{KolesnikovEMilman-Reilly}. 

Indeed, the proof there is based on the generalized Reilly formula, which states that for any compactly supported smooth function $u$ on $M$,
\begin{eqnarray*}
& & \int_M \brac{(\Delta_{g,\mu} u)^2 - \norm{\nabla^2 u}^2_{HS} - \scalar{\Ric_{g,\mu} \; \nabla u , \nabla u}} d\mu \\
& = &
\int_{\partial M} H_\mu u_n^2 d\mu_{\partial M} - 2 \int_{\partial M} \scalar{\nabla_{\partial M} u_n , \nabla_{\partial M} u} d\mu_{\partial M}  + \int_{\partial M} \scalar{\II_{\partial M} \;\nabla_{\partial M} u,\nabla_{\partial M} u} d\mu_{\partial M} ~,
\end{eqnarray*}
where $u_n$ denotes the partial derivative of $u$ in the direction of the outer normal $n$ to $\partial M$, $\nabla_{\partial M}$ denotes the induced connection on $\partial M$, $\II_{\partial M}$ denotes the second fundamental form, $H_\mu = tr(\II_{\partial M}) - \scalar{\nabla V,n}$ denotes the generalized weighted mean-curvature, and $\mu_{\partial M} = \exp(-V) d\vol_{\partial M}$ is the boundary measure. 
In particular, when $u$ in addition satisfies vanishing Neumann boundary conditions $u_n \equiv 0$, we see that when $M$ is locally-convex then
\[
\int_M (\Delta_{g,\mu} u)^2  d\mu \geq \int_M  \scalar{\Ric_{g,\mu} \; \nabla u , \nabla u} d\mu .
\]

Recall that we are also assuming that $f \in C^1(M)$, and apply all of the above to the sequence $u_k \in C_{c,Neu}^\infty(M)$. We obtain, after integrating by parts and applying Cauchy-Schwarz as before, that
\begin{align*}
\brac{\int_M f \Delta_{g,\mu} u_k d\mu}^2 & = \brac{\int_M \scalar{\nabla f ,\nabla u_k} d\mu}^2 \\
& \leq \int_M \scalar{\Ric_{g,\mu} \; \nabla u_k, \nabla u_k} d\mu \int_M \scalar{\Ric_{g,\mu}^{-1} \; \nabla f, \nabla f} d\mu \\
& \leq \int_M (\Delta_{g,\mu} u_k)^2  d\mu  \int_M \scalar{\Ric_{g,\mu}^{-1} \; \nabla f, \nabla f} d\mu .
\end{align*}
Letting $k \rightarrow \infty$ and invoking the convergence $-\Delta_{g,\mu} u_k \rightarrow f$ in $L^2(\mu)$, we deduce
\[
\int_M f^2 d\mu \leq \int_M \scalar{\Ric_{g,\mu}^{-1} \; \nabla f, \nabla f} d\mu ,
\]
as asserted. This concludes the proof. 
\end{proof}

Finally, we provide a justification of the Refined Brascamp--Lieb Theorem \ref{thm:refined-BL} for a Hessian metric $g = D^2 \Phi$ obtained from the optimal transport map $x \mapsto \nabla \Phi(x)$ between two log-concave measures. Recall that the challenge here lies in that we have no way of ensuring that the resulting manifold satisfies the Geometric Convexity Assumptions (e.g. completeness or geodesic convexity), which are required for invoking most of the tools we employ in this work.

\begin{thm} \label{thm:appendix-refinedBL}
Let $\mu,\nu$ denote two log-concave probability measures on $\Real^d$.  
Assume that $\mu = \exp(-V) dx$ on $\Real^d$ and that $\nu = \exp(-W) dx$ is supported on the convex set $\Omega \subset \Real^d$. Assume that both densities are positive and smooth on their corresponding supports, and that in addition either $D^2 V > 0$ or $D^2 W > 0$ there. 
Let $\nabla \Phi$ denote the optimal transport map pushing forward $\mu$ onto $\nu$. Then for all $f \in C^1(\Real^d)$,
\[
\Var_{\mu} (f) \le  \int \big\langle  Q^{-1} \nabla f, \nabla f   \big \rangle \ d \mu  ,
\]
where
\[
Q = \frac{1}{2} D^2 V + \frac{1}{2} D^2 \Phi \cdot D^2 W(\nabla \Phi) \cdot D^2 \Phi
+ \frac{1}{4d} \bigl( \nabla V - D^2 \Phi \cdot \nabla W(\nabla \Phi) \bigr) \otimes \bigl( \nabla V - D^2 \Phi \cdot \nabla W(\nabla \Phi)\bigr).
\]
\end{thm}
\begin{proof}
Note that by Theorem \ref{thm:Kol-compute} and Lemma \ref{lem:H}, we know that
\begin{equation} \label{eq:Q-bound}
\Ric_{g,\mu} \ge Q .
\end{equation}
Under the additional assumption that on the corresponding supports
\begin{equation} \label{eq:additional-assumption}
C > D^2 V > c >0 \; , \; C > D^2 W > c > 0 ,
\end{equation}
it was shown in \cite{KolesnikovGlobalSobolevCaffarelli} that the potential $\Phi$ satisfies
\[
\Delta(c,C) > D^2 \Phi > \delta(c,C)>0 ,
\]
implying that $(\Real^d,g = D^2 \Phi)$ is complete (oriented and without boundary), and hence $(\Real^d,g,\mu)$ satisfies the Geometric Convexity Assumptions. The assertion then follows by Theorem \ref{thm:BL-appendix} and the estimate (\ref{eq:Q-bound}). 
 
The general case is obtained by applying an approximation procedure which we outline below; further details may be found in \cite[Lemma 5.2]{KlartagKolesnikov-EigenvalueDistribution}. Without loss of generality we assume that $f$ is a smooth function supported on a closed ball $B_R$.
Choosing an appropriate sequence of smooth $V_n, W_n$ satisfying (\ref{eq:additional-assumption}) and approximating $V,W$ (respectively) and their derivatives up to second order locally uniformly, we may ensure that $D^2 \Phi_n \to D^2 \Phi$ locally uniformly. It follows from the initially treated case that
\begin{equation} \label{eq:almost}
\Var_{\mu_n} (f) \le  \int \big\langle  Q_n^{-1} \nabla f, \nabla f   \big \rangle \ d \mu_n,
\end{equation}
where $\mu_n = \exp(-V_n(x)) dx$ and
\begin{align*}
Q_n & := \frac{1}{2} D^2 V_n + \frac{1}{2} D^2 \Phi_n \cdot D^2 W_n(\nabla \Phi_n) \cdot D^2 \Phi_n
\\&+ \frac{1}{4d} \bigl( \nabla V_n - D^2 \Phi_n \cdot \nabla W_n(\nabla \Phi_n) \bigr) \otimes \bigl( \nabla V_n - D^2 \Phi_n \cdot \nabla W_n(\nabla \Phi_n)\bigr) .
\end{align*}

Next, we observe that the set $\Omega_r := \overline{\cup_n \nabla \Phi_n(B_R)} \subset \Omega$ is compact, and moreover lies in the interior of $\Omega$. 
Indeed, arguing in the contrapositive, we obtain a sequence $B_R \ni x_n \to x_0$  so that $\abs{\nabla \Phi_n(x_n)} \to \infty$ if the former statement is violated, and $\nabla \Phi_n(x_n) \to \partial \Omega$ if the latter one is. However,
\begin{align*}
|\nabla \Phi(x_0) -  \nabla \Phi_n(x_n)| & \le |\nabla \Phi(x_0) - \nabla \Phi_n(x_0)| + | \nabla \Phi_n(x_0) - \nabla \Phi_n(x_n)| \\
& \le |\nabla \Phi(x_0) -  \nabla \Phi_n(x_0)|  +  \sup_{x \in B_R} \|D^2 \Phi_n(x) \| |x_0- x_n| \to 0 .
\end{align*}
This already yields a contradiction in the former case. As for the latter one, we deduce by the already established compactness that $\nabla \Phi(x_0) \in \partial \Omega$, in violation of the strong convexity  (and smoothness) of the transport potential $\Phi$, which ensures that $\Img(\nabla \Phi)$ must be an open (convex) set \cite{Gromov-KahlerManifolds}. 
Recalling our assumption that either $D^2 V > 0$ or $D^2 W > 0$, we can now apply the uniform $C^2$ convergence $V_n \to V$ on $B_R$ and $W_n \to W$ on the compact $\Omega_r$, and deduce that the positive definite matrices $Q_n$ are in fact uniformly positively bounded below on $B_R$.
The proof of the general case would be concluded by applying the dominated convergence theorem in (\ref{eq:almost}), provided we knew that  $Q_n \to Q$ almost everywhere with respect to the Lebesgue measure (after perhaps passing to a subsequence). This is obvious, except for the convergence of the problematic terms $\nabla W_ n(\nabla \Phi_n), D^2 W_n(\nabla \Phi_n)$. To see the convergence of e.g. $D^2 W_n(\nabla \Phi_n)$, write
\begin{align*}
& \int_{B_R} |\partial^2_{x_i x_j} W(\nabla \Phi) - \partial^2_{x_i x_j}  W_{n}(\nabla \Phi_n)| d \mu_n \\
& \le  \int_{B_R} |\partial^2_{x_i x_j} W(\nabla \Phi ) - \partial^2_{x_i x_j}  W( \nabla \Phi_n )| d \mu_n
 +   \int_{\nabla \Phi_n(B_R)}  |\partial^2_{x_i x_j}  W- \partial^2_{x_i x_j}  W_{n}| d \nu_n.
\end{align*}
Using again that $\Omega_r \supset\nabla \Phi_n(B_R)$ is compact and lies in the interior of $\Omega$, we see that the right-hand side converges to zero, and hence (up to passing to a subsequence) $\partial^2_{x_i x_j}  W_{n}(\nabla \Phi_n) \to \partial^2_{x_i x_j}  W(\nabla \Phi)$ almost everywhere. The convergence of $\nabla W_ n(\nabla \Phi_n)$ is handled identically, thereby completing the proof. 
\end{proof}

\bibliographystyle{plain}
\bibliography{../../../ConvexBib}

\end{document}